\documentclass[11pt]{article}
\usepackage{a4wide}

\usepackage[a4paper, margin=2.3cm]{geometry}
\usepackage{amsmath,amssymb,amsthm, comment, enumitem, array}
\usepackage{color}
\usepackage{parskip}
\usepackage{hyperref}

\usepackage{parskip}
\usepackage{setspace}
\usepackage{graphicx}
\usepackage{mathtools}
\usepackage[thinc]{esdiff}
\usepackage[euler]{textgreek}

\usepackage{babel}

\usepackage[capitalise, nameinlink]{cleveref}
\crefname{equation}{}{}
\crefname{figure}{{\sc Figure}}{{\sc Figure}}
\crefname{subsection}{Subsection}{Subsections}

\newtheorem{theorem}{Theorem}[section]
\newtheorem{proposition}[theorem]{Proposition}
\newtheorem{remark}[theorem]{Remark}
\newtheorem{corollary}[theorem]{Corollary}
\newtheorem{lemma}[theorem]{Lemma}

\newtheorem*{definition*}{Definition}

\hypersetup{
    colorlinks=true,
    linkcolor = blue,
    citecolor=magenta,
    urlcolor=cyan,
}

\def\F{\mathcal{F}}
\def\M{\mathcal{M}}

\def\Fbb{\mathbb{F}}

\def\Rbb{\mathbb{R}}
\def\Cbb{\mathbb{C}}

\def\Nbb{\mathbb{N}}
\def\R{\mathbb{R}}

\def \R{{\mathbb R}}

\def \F{{\mathbb F}}

\def \1{\textbf{1}}

\def\Ebf{\textbf{E}}
\def\SL{SL}

\def\l({\left(}
\def\r){\right)}

\def\lv{\left\vert}
\def\rv{\right\vert}

\def\lo{\left\{}
\def\ro{\right\}}

\def\lV{\left\Vert}
\def\rV{\right\Vert}

\def\lf{\left\lfloor}
\def\rf{\right\rfloor}

\def\rar{\rightarrow}

\begin{document}
\title{Packing sets under finite groups via algebraic incidence structures}
\author{Norbert Hegyv\'{a}ri\thanks{E\"otv\"os University and associated member of Alfréd Rényi Institute, Hungary. Email: hegyvari@renyi.hu}\and Le Quang Hung\thanks{Hanoi University of Science and Technology, Vietnam. Email: hung.lequang@hust.edu.vn}\and Alex Iosevich\thanks{Department of Mathematics, University of Rochester, USA. Email: iosevich@math.rochester.edu}\and Thang Pham\thanks{Institute for Mathematics and Interdisciplinary Sciences, Xidian University. \newline
\hspace*{0.5cm} Email: {\tt thangpham.math@gmail.com}}}
\date{}
\maketitle

\begin{abstract}
Let $G$ be a finite group acting on a vector space $V = \mathbb{F}_p^n$ over a prime field. Given finite sets $S \subset G$ and $E \subset V$, we study the restricted orbit union $S(E) = \bigcup_{g\in S} g(E)$ and establish quantitative lower bounds for $|S(E)|$ in terms of $|S|$, $|E|$, and natural structural conditions. This finite field packing problem has connections to distance geometry, configuration counting, and expanding graphs. For $G = SL_2(\mathbb{F}_p)$ acting on $\mathbb{F}_p^2$, we prove that $$|S(E)| \gg \min\left\lbrace p^2, \frac{|S||E|}{p^2}\right\rbrace,$$ which is sharp. Under geometric non-concentration conditions on $E$ and subgroup-avoidance hypotheses on $S$, we obtain a power-saving improvement of the form $$|S(E)|\gg \min \left\lbrace p^2, ~\max\left\lbrace\frac{|S||E|}{pk}, ~\frac{|S|^{\frac{1}{2}}|E|}{p^{\frac{1-\epsilon}{2}}k^{\frac{1}{2}}}\right\rbrace \right\rbrace,$$ where $k$ bounds the radial multiplicity of $E$. For small sets $|E| \leq p$, we establish optimal bounds using weighted incidence theory. Analogous results are proved for the first Heisenberg group $\mathbb{H}_1(\mathbb{F}_p)$ acting on $\mathbb{F}_p^3$. Our approach reformulates the problem as an incidence question in a bipartite action graph. The proofs combine Fourier analytic techniques, energy estimates, point-line incidence bounds, and area-energy inequalities for skew dot products. The methods extend classical sum-product type problems and incidence theory to noncommutative group actions.
\end{abstract}


\textbf{MSC-2020}: 11T30, 51B05

\tableofcontents
\section{Introduction}

Let $E$ be a Borel set in $\mathbb{R}^n$ and let $S$ be a set of maps from $\mathbb{R}^n$ to $\mathbb{R}^n$. We define 
\begin{align*}
    S(E):=\bigcup_{f\in S}f(E)=\{f(y): f\in S, y\in E\}\subset \R^n.
\end{align*}
The packing problem asks if it is possible for a set of zero $n$-dimensional Lebesgue measure to contain the image $f(E)$ for all $f\in S$. The study of this problem has a well-known history in the literature, for example, there are sets in the plane of zero Lebesgue measure containing  a line segment of unit length in every direction (see \cite{Besicovitch20} and \cite{Besicovitch28}), a circle of radius $r$ for all $r>0$ (see \cite{BesicovitchRado68} and \cite{Kinney68}), or a circle centered at $x$ for all $x$ on a given straight line (see \cite{Talagrand80}). In another direction, the question of finding conditions on $S$ and $E$ such that $S(E)$ has positive Lebesgue measure has also received a lot of attention. Bourgain \cite{Bourgain85} and independently Marstrand \cite{Marstrand87} proved that given a set of circles in the plane, if the centers form a set of positive Lebesgue measure, then the union of circles also has positive Lebesgue measure. Wolff \cite{Wolff97} strengthened this result by showing that if the set of centers has Hausdorff dimension $s$, $0<s\leq 1$, then the dimension of the set of union of circles is at least $1+s$. For the most recent progress, we refer the reader to \cite{Iosevichetal} and references therein. 

Let $\mathbb{F}_p$ be a prime field of order $p$, where $p$ is an odd prime. Let $G$ be a finite group acting on $V=\mathbb{F}_p^n$.
Given finite sets $S\subset G$ and $E\subset V$, define the restricted orbit union set
\[
S(E):=\{gy:\ g\in S,\ y\in E\}=\bigcup_{g\in S} g(E).
\]
In this setting, the packing question translates into quantitative lower bounds for the size of $S(E)$. This paper provides such bounds for $|S(E)|$ in terms of $|S|$, $|E|$, and natural non-concentration conditions.

Beyond its intrinsic expanding flavor, it is worth mentioning that restricted orbit union expansion implies several configuration-counting consequences, including pinned triangle congruence and distance questions. We briefly indicate this connection as follows. When $S$ is a set of rigid-motions, i.e. a set of maps preserving distances, the author of \cite{Pham2024} proved a lower bound on the number of congruence classes of triangles with one fixed side-length in the plane. By setting $S=\{g\}\times E'$, for some rotation $g\in O(n)$ and $E'\subset \mathbb{F}_p^n$, and $S(E)=\{gy-x\colon y\in E, x\in E'\}$, the lower bound $|S(E)|\gg p^n$ implies that the distance set between $gE$ and $E'$, denoted by $\Delta(gE, E')$, covers a positive proportion of all distances in $\mathbb{F}_p$. The authors of \cite{PS} proved that under the conditions $|E|, |E'|\ge p$, for almost every rotation $g$ in the plane $\mathbb{F}_p^2$, $|gE-E'|\gg p^2$, and obtained $|\Delta(gE, E')|\gg p$. Note that the Erd\H{o}s-Falconer distance conjecture in two dimensions corresponds to the case $g=id$. 


\paragraph{Notations:} Throughout the paper, we will write $X \ll_\alpha Y$ if $X \le CY$, where $C>0$ is a constant depending on $\alpha$. If it is clear from the context what $C$ should depend on, we may write only $X \ll Y$. If $X \ll Y $ and $Y\ll X$, we write $X \sim Y$. Furthermore, $X \lesssim Y$ if $X \ll \log_2p\cdot Y$.  

\subsection{The special linear group \texorpdfstring{$G=SL_2(\mathbb{F}_p)$}{SL2(Fp)}}
We first start with some observations.

{\bf Observation 1:} Since each $\theta\in SL_2(\mathbb{F}_p)$ acts bijectively on $\mathbb{F}_p^2$, we have $|S(E)|\ge |E|$. One might hope that $|S(E)|\ge |S|$, which implies $|S(E)|\ge |S|^{\frac{1}{2}}|E|^{\frac{1}{2}}$. However, the estimate $|S(E)|\ge |S|$ might not be true. For example, take $E=\{(0, 1)\}$ and $S$ being a set of matrices $\theta$ in $SL_2(\mathbb{F}_p)$ such that $\theta (0, 1)=(1, 0)$, then, by Lemma \ref{lmmay1}, $S$ can be as large as $\sim p$, and in this case, one has $|S(E)|=1$.

There are sets $S$ and $E$ such that $|S(E)|=|S|^{\frac{1}{2}}|E|^{\frac{1}{2}}$. Indeed, let $A\subset B $ be two subgroups of $\Fbb_p^*$. For each $[x] \in B/A$, fix $x' \in [x]$. Let $S_{[x]}$ be a subset of $\lo \theta \in \SL_2 \l( \Fbb_p \r) \colon \theta (0,1) =(0,x') \ro$ such that $\lv S_{[x]} \rv = \lv B\rv$. Let $E= \lo  (0,a) \colon a\in A \ro$. Then, we have
    \[ S_{[x]} (E) = \{ 0\} \times [x] .\]
    Therefore, if we choose $S = \bigcup_{[x]\in B/A} S_{[x]}  $, we have $\lv S \rv = \lv B\rv\lv B/A\rv$ and
    \[ S (E) = \bigcup_{[x]\in B/A} \{ 0\} \times [x] =\{ 0\} \times B. \]
    Then $\lv S (E)\rv = \lv B\rv = (\lv B\rv \lv B/A\rv)^{\frac{1}{2}} \lv A\rv^{\frac{1}{2}} = \lv S \rv^{\frac{1}{2}} \lv E \rv^{\frac{1}{2}} $.

{\bf Observation 2:} There are sets $S$ and $E$ such that $\lv S(E) \rv \sim \frac{\lv S\rv \lv E \rv}{p^2}>|E|$. Indeed, fix $0<\epsilon<1$, let $E$ be the set of points on the line $\{y=0\}$. Let $S \subset \SL_2 \l( \Fbb_p \r)$ be a set of all matrices $\theta$ such that $\theta (1,0) \in E'$, where $E'$ is a set of $p^{1+\epsilon}$ points on $p^{\epsilon}$ lines passing through the origin. Then, by Lemma \ref{lmmay1}, we have $\lv S\rv \sim p^{2+\epsilon}$. Therefore,
\[ \lv S(E)\rv =p^{1+\epsilon} \sim \frac{p^{2+\epsilon}\cdot p}{p^2} \sim \frac{\lv S\rv \lv E\rv}{p^2}. \]

In the first result, we prove that the lower bound $|S||E|/p^2$ actually holds for all sets $S$ and $E$. 
\begin{theorem}\label{thm2}
    Let $E\subset \mathbb{F}_p^2$ and $S\subset SL_2(\mathbb{F}_p)$. We have 
    \[|S(E)|\gg \min \left\lbrace p^2, ~\frac{|S||E|}{p^2} \right\rbrace.\]
\end{theorem}
It follows from this theorem that if $|S||E|\gg p^4$ then $|S(E)|\gg p^2$. This condition is optimal in the sense that for all $\epsilon>0$ there exist sets $S$ and $E$ with $|S||E|\gg p^{4-\epsilon}$ such that $|S(E)|=o(p^2)$. Indeed, for $\epsilon>0$, by choosing $p$ large enough, we can find a cyclic subgroup $A$ of $\mathbb{F}_p^*$ with $|A|\sim p^{1-\epsilon}$. Let $S$ be the set of matrices $\theta$ in $SL_2(\mathbb{F}_p)$ such that $\theta(0, 1)\in \{(-x, 0)\colon x\in A \}$. Each such matrix is of the form
\[\begin{bmatrix}
    *&-x\\
    x^{-1} &0
\end{bmatrix},\]
and $|S|\sim p^{2-\epsilon}$. We now let $E$ to be the set of points of the form $(y, *)$ with $y\in A$ and $*\in \mathbb{F}_p$. Then $|E|\sim p^{2-\epsilon}$. Moreover, $S(E)$ is covered by the lines of the from $y=\lambda$ with $\lambda\in A$. So $|S(E)|\ll p^{2-\epsilon}$.

Under subgroup non-concentration for $S$ and geometric non-concentration for $E$, the next theorem presents quantitatively stronger expansion bounds than those in Theorem~\ref{thm2}.

\begin{theorem}\label{thm555}
For any $\gamma\in (0, 1)$, there exists $\epsilon=\epsilon(\gamma)>0$ such that the following holds. 
     Let $E\subset \mathbb{F}_p^2$ and $S\subset SL_2(\mathbb{F}_p)$. Assume $p$ is a sufficiently large prime, $S$ is a symmetric subset of $\SL_2 (\Fbb_p )$ such that $p^{\gamma} < |S| < p^{3-2\gamma } $ and $ |S \cap gH| < p^{-\frac{\gamma}{2}}|S|$ for any subgroup $H \subsetneq \SL_2 (\Fbb_p) $ and $g\in \SL_2(\Fbb_p)$, and any line through the origin contains at most $k$ points from $E$, then 
     \[|S(E)|\gg \min \left\lbrace p^2, ~\max\left\lbrace\frac{|S||E|}{pk}, ~\frac{|S|^{\frac{1}{2}}|E|}{p^{\frac{1-\epsilon}{2}}k^{\frac{1}{2}}}\right\rbrace \right\rbrace.\]
\end{theorem}

\paragraph{Sharpness:} We observe that conditions on $S$ in the above theorem are natural for improvements. Indeed, let $\ell_1$  and $\ell_2$ be two lines passing through the origin and $S$ be the set of all $\theta \in \SL_2(\Fbb_p)$ such that $\theta (\ell_1) =\ell_2$. Then, let $E= \ell_1$, we have $S(E)=\ell_2$, and by Lemma \ref{lmmay1}, $|S|\sim p^2$, so
    \[ |S(E)| = p \sim \frac{|S|^{\frac{1}{2}}|E|}{p}. \]
    Moreover, we have $S=gH$, where $H$ is the group of matrices $\theta \in \SL_2(\Fbb_p)$ such that $\theta (\ell_1) =\ell_1$ and $g\in \SL_2(\Fbb_p)$ is a matrix such that $g(\ell_1) =\ell_2$. 
 
We now turn our attention to the case of small sets. If $S$ and $E$ are arbitrary sets of small size, then based on Observation 1 and Observation 2, one can guess of the existence of sets such that the size of $S(E)$ is the same as the trivial lower bound $|E|$. For example, let $A$ be a subgroup of $\mathbb{F}_p^*$, and let $E=\{\lambda (1, 0)\colon \lambda\in A\}$, and $S$ be a set of matrices 
$\theta$ such that $\theta$ maps at least one point in $E$ to $(0, 1)$. Depending on the form of $p$, one can choose $E$ of arbitrary small size, and the size of $S$ can also be chosen arbitrary large in the range $(0, p|E|)$ such that $|S(E)|=|E|$. 

In the setting of small sets, we prove the following optimal result. 
\begin{theorem}\label{thm345} 
  Let $S \subseteq \SL_2 \l( \Fbb_p \r)$, and $E \subseteq \Fbb_p^2\setminus \lo (0, 0)\ro$ be such that $|E| \le p$, any line passing through the origin contains at most $k_1$ points from $E$, and $E$ determines at most $k_2$ distinct directions through the origin. Then,
  \[ \lv S(E)\rv \gtrsim \min \lo  \frac{\lv E\rv \lv S\rv^{\frac{1}{2}}}{k_1^{\frac{1}{2}}k_2^{\frac{1}{2}}}, \frac{\lv E\rv^{\frac{1}{2}} \lv S\rv^{\frac{1}{2}}}{k_1^{1/4}}, \frac{\lv E\rv \lv S\rv^{\frac{1}{2}}}{k_1}, \frac{\lv E\rv^2}{k_1} \ro .\]
Here, by ``directions" we mean the number of lines passing through the origin that are required to cover the whole set $E$.
\end{theorem}
In order to have $|S(E)|> |E|$, one would need the conditions that 
\[|E|\gg k_1, ~~|S|\ge \max \left\lbrace k_1^2, k_1^{1/2}|E|, k_1k_2  \right\rbrace.\]

The lower bound of Theorem \ref{thm345} is attainable from the following example.

For a constant $0<c<1$, let $\lo \ell_i \ro_{i=1}^{\frac{p}{c}}$ be a family of $\frac{p}{c}$ lines passing through the origin. For each $1\le i \le \frac{p}{c}$, choose $a_i \in \ell_i \setminus \lo (0,0) \ro$, and set $E = \lo a_i \colon 1 \le i \le \frac{p}{c} \ro$, $S = \SL_2(\Fbb_p)$. Then we have $k_2=|E|=\frac{p}{c} ,k_1=1$, and $|S(E)|=p^2-1$. 
    Thus, 
     \[ \lv S(E)\rv \sim  \min \lo  \frac{\lv E\rv \lv S\rv^{\frac{1}{2}}}{k_1^{\frac{1}{2}}k_2^{\frac{1}{2}}}, \frac{\lv E\rv^{\frac{1}{2}} \lv S\rv^{\frac{1}{2}}}{k_1^{1/4}}, \frac{\lv E\rv \lv S\rv^{\frac{1}{2}}}{k_1}, \frac{\lv E\rv^2}{k_1} \ro .\]

As in Theorem \ref{thm555}, we now provide an $\epsilon$-improvement under non-concentration conditions.

\begin{theorem}\label{thm3'}
For any $\gamma\in (0, 1)$, there exists $\epsilon=\epsilon(\gamma)>0$ such that the following holds. 
  Let $p$ be a sufficiently large prime, and $E \subseteq \Fbb_p^2\setminus \lo (0, 0)\ro$ such that $|E| \le p$, any line passing through the origin contains at most $k_1$ points from $E$ and $E$ determines at most $k_2$ distinct directions through the origin. Let $S \subseteq \SL_2 \l( \Fbb_p \r)$ be a symmetric subset such that $p^{\gamma} < \lv S\rv < p^{3-2\gamma} $ and $\lv S \cap gH\rv < p^{-\frac{\gamma}{2}} \lv S\rv $ for any subgroup $H \subsetneq \SL_2 \l( \Fbb_p \r)$ and $g \in \SL_2 \l( \Fbb_p \r)$. Then,
  \[ \lv S(E)\rv \gtrsim \min \lo  \frac{\lv E\rv \lv S\rv^{\frac{1}{2}}p^{\epsilon /2}}{k_1^{\frac{1}{2}}k_2^{\frac{1}{2}}}, \frac{\lv E\rv^{\frac{1}{2}} \lv S\rv^{\frac{1}{2}}p^{\epsilon /2}}{k_1^{1/4}}, \frac{\lv E\rv \lv S\rv^{\frac{1}{2}}p^{\epsilon /2}}{k_1}, \frac{\lv E\rv^2}{k_1} \ro .\]
\end{theorem}

Our approach to quantitative lower bounds for $|S(E)|$ is to recast the orbit-union problem as an incidence problem attached to the group action.
Instead of working directly with
\[
S(E)=\{gy:\ g\in S,\ y\in E\},
\]
we encode the relation $gy=x$ as incidences in a bipartite \emph{action-incidence graph} whose left vertices are pairs $(x,y)$ and whose right vertices are group elements $g\in S$.
This reformulation reduces expansion estimates for $|S(E)|$ to proving \emph{incidence discrepancy} bounds, i.e. upper bounds for
\[
\Bigl|I(A\times B,S)-\tfrac{|A||B||S|}{p^{2}}\Bigr|
\]
in terms of natural structural parameters of $A,B$, and $S$.

For $G=\SL_2(\F_p)$ acting on $\F_p^2$, we first establish a sharp universal estimate of the form
\[
\Bigl| I(A\times B,S)-\tfrac{|A||B||S|}{p^{2}} \Bigr|
\ \ll\ p\sqrt{|A||B||S|}+|S|
\]
(Theorem~\ref{incidence}).
The proof is Fourier-analytic: after expanding the incidence constraint and applying Cauchy--Schwarz, the key step is a geometric reduction using the fact that $\SL_2(\F_p)$ preserves the skew dot product (equivalently, signed areas), cf.\ Lemma~\ref{lmmay2}.
This converts the main contribution to a bound for the ``area-energy'' quantity
\[
\#\{(x_1,x_2,y_1,y_2)\in A^2\times B^2:\ x_1\cdot x_2^\perp=y_1\cdot y_2^\perp\},
\]
which is controlled by known $L^2$ estimates for the skew dot-product/determinant function.
Applying the resulting incidence bound with $A=S(E)$ and $B=E$, and comparing with the trivial lower bound
$I(S(E)\times E,S)\ge |S||E|$, implies the one-step expansion bound in Theorem~\ref{thm2}.

To obtain stronger bounds when additional structure is available, we refine the incidence analysis via an \emph{energy decomposition}.
Under geometric non-concentration on $E$ (at most $k$ points on any line through the origin) and subgroup non-concentration for $S$, we prove an $\epsilon$-improved incidence estimate (Theorem~\ref{thm: 4.12})
\[
I(S(E)\times E,S)\ \ll\ \frac{|S(E)|^{1/2}|E||S|}{p}
\ +\ k^{\frac{1}{4}}p^{\frac{1-\epsilon}{4}}|S(E)|^{\frac{1}{2}}|E|^{\frac{1}{2}}|S|^{\frac{3}{4}}.
\]
The mechanism is a Cauchy--Schwarz reduction to a collision count, which naturally introduces two energies:
(i) a geometric area-energy of $E$, controlled in terms of $k$, and
(ii) the multiplicative energy $E(S,S)=\#\{(a,b,c,d)\in S^4:\ ab=cd\}$.
While $E(S,S)$ can be as large as $\sim |S|^{3}$ in general, Bourgain--Gamburd $L^2$-flattening yields the power-saving bound
$E(S,S)\ll |S|^{3}p^{-\epsilon}$ under subgroup avoidance \cite{B.A.08}.
Incorporating this saving into the refined incidence estimate produces the $\epsilon$-improved expansion statements in Theorems~\ref{thm555} and~\ref{thm3'}.

Finally, in the small-set regime $|E|\le p$, the Fourier methods become ineffective and we instead use incidence geometry.
We prove a weighted/multi-set extension of the Stevens--de Zeeuw point-line incidence theorem (Theorem~\ref{eqn:IPLMS}), obtained via a dyadic decomposition of multiplicities.
This weighted incidence input is then used to bound skew dot-product energies for small configurations, leading to sharp expansion estimates in terms of the \emph{radial multiplicity} $k_1$ (maximum number of points of $E$ on a line through the origin) and the \emph{number of directions} $k_2$ determined by $E$, as in Theorems~\ref{thm345} and~\ref{thm3'}.

One might ask about the higher dimensional case. Let's assume $n=3$ for simplicity. It is not hard to check that if we have two triples $(x, y, z)$ and $(x', y', z')$ such that each forms an independent system, then there exists unique $\theta\in SL_3(\mathbb{F}_p)$ such that $\theta x=x'$, $\theta y=y'$, and $\theta z=z'$ if and only if $\det(x, y, z)=\det(x', y', z')$. Here $\det(x, y, z)$ is the determinant of the matrix with columns $x$, $y$, and $z$. If one wishes to apply the method in the two dimensions, then an estimate on the number of tuples $(x, y, z, x', y', z')\in E^6$, $E\subset \mathbb{F}_p^3$, such that $\det(x, y, z)=\det(x', y', z')$ is needed in the first step. This is the $L^2$-norm version of earlier results studied in \cite{covert, lavinh}. Notice that this framework only solves the case of large sets, and for the case of small sets in higher dimensions, it appears to be a hard problem due to the limited understanding of incidence bounds. We plan to address this in a subsequent paper.

\subsection{The first Heisenberg group 
\texorpdfstring{$G=\mathbb{H}_1(\mathbb{F}_p)$}{H1(Fp)}}

We now move to the case of the Heisenberg group. Let $\mathbb{F}_p$ be a prime field, we denote by $\mathbb{H}_1(\mathbb{F}_p)$ the first Heisenberg group over $\mathbb{F}_p$, i.e. the group of matrices of the form 
\[ [x, y, t]:=\begin{pmatrix}1&x&t\\
0&1&y\\
0&0&1\end{pmatrix}, ~x, y, t\in \mathbb{F}_p.\]
In the group $\mathbb{H}_1(\Fbb_p)$, we have 
\[[x, y, t]\cdot [x', y', t']=\left(x+x', y+y', t+t'+\frac{xy'-yx'}{2}\right).\]

This section uses the notation $X(E)$ in order to distinguish with the case of $SL_2(\mathbb{F}_p)$.

We first look at the following observations. 

{\bf Observation 3:} 
Let $X=\mathbb{H}_1(\Fbb_p)$ and let $E$ be the set of all points with the third coordinate belonging to a set of size $\alpha p$ in $\mathbb{F}_p$. A direct calculation gives $|X(E)|\le \alpha p^3$. This example tells us that in order to cover the whole space or a positive proportion of all elements in $\mathbb{F}_p^3$, the condition $|E|\gg p^3$ is needed. 

{\bf Observation 4:} There are sets $X$ and $E$ of arbitrary large such that $|X(E)|\ll |X|$. Let $A\subset\mathbb{F}_p$ be an arbitrary set, let $X$ be the set of matrices $[a,b,c]\in \mathbb{H}_1(\Fbb_p)$ with $a, c\in \mathbb{F}_p$ and $b\in A$, let $E$ be the set of points $(x, y, z)$ with $x\in \mathbb{F}_p,~ y\in A,$ and $z\in A$. Then we have $|X|=p^2|A|$ and $|E|=p|A|^2$. A direct computation shows that $|X(E)|\ll p^2|A|= |X|$. Thus, for any $0<\epsilon<2$, there exist sets $E\subset \mathbb{F}_p^3$ and $X\subset \mathbb{H}_1(\mathbb{F}_p)$ with $|E|, |X|\ge p^{3-2\epsilon}$ such that $|X(E)|\ll |X|$.

Let $\pi_{23}\colon \mathbb{F}_p^3\to \mathbb{F}_p^2$ defined by $\pi_{23}(x, y, z)=(y, z)$. The following example shows that if the pre-image set $\pi_{23}^{-1}(y, z)\cap E$ is of large size for all $(y, z)\in \pi_{23}(E)$, then the trivial bound $|E|$ might be best possible.

{\bf Observation 5:} There are sets $X$ and $E$ such that $|X(E)|\ll |E|$. Let $A\subset\mathbb{F}_p$ be an arithmetic progression. Let $E$ be a set in $\mathbb{F}_p^3$ such that each point in $E$ has the last two coordinates belonging to $A$. Assume that for each $(y, z)\in \pi_{23}(E)$, we have $|\pi_{23}^{-1}(y, z)\cap E|\sim p$. Then, for all sets $X$ containing of matrices of the form $[*, 1, *]\in \mathbb{H}_1(\Fbb_p)$, we have $|X(E)|\ll |E|$. 

Under a non-concentration condition, the main theorem in this section reads as follows. 

\begin{theorem}\label{main-H}
Let $\epsilon\ge 0$, $X$ be a subset of $\mathbb{H}_1(\Fbb_p)$, and $E$ be a set in $\mathbb{F}_p^3\setminus (\mathbb{F}_p^2\times \{0\})$. Assume that for each $(y, z)\in \mathbb{F}_p^2$, we have $|\pi_{23}^{-1}(y, z)\cap E|\le p^{1-\epsilon}$, then we have
\[|X(E)|\gg \min \left\lbrace p^3, ~\frac{|X||E|}{p^{3-\frac{\epsilon}{2}}} \right\rbrace.\]
\end{theorem}

Note that when $\epsilon=0$, it implies directly that 
\[|X(E)|\gg \min \left\lbrace p^3, ~\frac{|X||E|}{p^{3}} \right\rbrace,\]
which is sharp by Observation 3. 

As in the $\SL_2(\F_p)$ setting, we prove Theorem~\ref{main-H} by encoding the action relation
$\theta y=x$ as incidences in an associated bipartite action-incidence graph. However, the
underlying invariant structure is different.  Writing
\[
[a,b,c]:=\begin{pmatrix}1&a&c\\ 0&1&b\\ 0&0&1\end{pmatrix}\in \mathbb{H}_1(\F_p),
\qquad
[a,b,c](x,y,z)=(x+ay+cz,\; y+bz,\; z),
\]
we see that the third coordinate is preserved; accordingly we restrict to points with $z\neq 0$
(and likewise in Theorem~\ref{H-inci}).

The key step in the incidence analysis is to understand when a \emph{single} group element
$[a,b,c]$ can simultaneously map two points $(x,y,z),(u,v,w)$ to two targets
$(x',y',z'),(u',v',w')$.  From the explicit action one immediately gets $z'=z$, $w'=w$ and
\[
b=\frac{y'-y}{z}=\frac{v'-v}{w},
\]
so compatibility forces the bilinear constraint
\begin{equation}\label{bracket}
z\,(v'-v)=w\,(y'-y)
\qquad
\text{(equivalently } y\,w+z\,v'=y'\,w+v\,z\text{)},
\end{equation}
together with $z=z'$ and $w=w'$.  After a Fourier expansion of the incidence relation and an
application of Cauchy--Schwarz, the discrepancy reduces to bounding a \emph{weighted} count of
quadruples satisfying this bracket equation, where the weights are precisely the multiplicities of
fibers of the projection $\pi_{23}(x,y,z)=(y,z)$.  A weighted orthogonality estimate
(Lemma~\ref{weight-H}) combined with the fiber non-concentration hypothesis
$|\pi_{23}^{-1}(y,z)\cap B|\le p^{1-\epsilon}$ gives
\[
\Bigl|I(A\times B,X)-\tfrac{|A||B||X|}{p^{3}}\Bigr|
\ \ll\ p^{\frac{3-\epsilon}{2}}\,|A|^{\frac{1}{2}}|B|^{\frac{1}{2}}|X|^{\frac{1}{2}},
\]
i.e. Theorem~\ref{H-inci}, with degenerate configurations handled by a
separate analysis.
Applying the incidence bound with $A=X(E)$ and $B=E$ then gives Theorem~\ref{main-H}.
Finally, Observation~5 shows that fiber control is essentially necessary: if many points of $E$
lie in single $\pi_{23}$-fibers (of size $\sim p$), even large $X$ may fail to expand beyond the
trivial scale.  We therefore do not pursue small-set analogues in $\mathbb{H}_1(\F_p)$ here, as the
corresponding three-dimensional constraint of (\ref{bracket}) lead to substantially more intricate incidence
geometry.


\section{Incidence structures spanned by \texorpdfstring{$SL_2(\mathbb{F}_p)$}{SL2(Fp)}}
Let $x,y \in \Fbb_p^2$ and $\theta \in \SL_2 \l( \Fbb_p \r) $, we say $(x,y)$ is incident to $\theta$ if $\theta y=x$. Let $P=A\times B \subseteq \Fbb_p^2 \times \Fbb_p^2$ and $S\subseteq \SL_2 \l( \Fbb_p \r) $, we denote the number of incidences between $P$ and $S$ by $I(P,S)$.

In this section, we prove the following incidence bound. 

\begin{theorem}\label{incidence} Let $P=A\times B \subseteq \mathbb{F}_p^2\times \mathbb{F}_p^2$ and $S \subseteq \SL_2 \l( \Fbb_p \r) $. Then, we have
    \[\lv I(P,S)- \frac{\lv P\rv \lv S\rv}{p^2} \rv \ll p\sqrt{\lv S\rv \lv P\rv}+|S|.\]
Moreover, let $k_A$ and  $k_B$ be the maximal number of points from $A$ and $B$ on a line passing through the origin, respectively. Assume that $\min\{k_A, k_B\}= k$,
then
 \[\lv I(P,S)- \frac{\lv P\rv \lv S\rv}{p^2} \rv \ll p^{\frac{1}{2}}k^{\frac{1}{2}}\sqrt{\lv S\rv \lv P\rv}+|S|.\]
\end{theorem}

We have some comments on this theorem. 
\begin{enumerate}
    \item This theorem is sharp, i.e. there are sets $P$ and $S$ such that the upper bound is attained. 

     Let $P= A \times B$ where $A,B$ are the sets of points on two lines through the origin. Let $S$ be the set of all matrices $\theta$ in $\SL_2 \l( \Fbb_p \r)$ such that $\theta (B)=A$. Then, we have $\lv S\rv =p(p-1), \lv P\rv =p^2$, and
    \[ I(P,S) = p(p-1) + p(p-1)^2.  \]
    Hence,
    \[ \lv I(P,S) - \frac{\lv P\rv \lv S\rv}{p^2} \rv = p(p-1)^2 \sim p\sqrt{\lv P\rv \lv S\rv} + \lv S\rv .\]
    
    \item The same result holds for general sets $P\subset \mathbb{F}_p^4$ instead of sets of Cartesian product structures.
\end{enumerate}
If we assume some structural conditions on $S$, the upper bound of the above theorem can be improved further. 

\begin{theorem}\label{thm: 4.12}
For $\gamma\in (0, 1)$, there exists $\epsilon=\epsilon(\gamma)>0$ such that the following holds.
    Let $p$ be a sufficiently large prime. Let $P=A\times B \subseteq (\Fbb_p^2 \times \Fbb_p^2 ) \setminus \lo (0, 0, 0, 0)\ro$, and let $S$ be a symmetric subset of $\SL_2 (\Fbb_p )$ such that $p^\gamma < |S| < p^{3-2\gamma } $ and $ |S \cap gH| < p^{\frac{-\gamma}{2}}|S|$ for any subgroup $H \subsetneq \SL_2 (\Fbb_p) $ and $g\in \SL_2(\Fbb_p)$. Then, we have 
    \[ I(P,S) \ll \frac{|A|^{\frac{1}{2}}|B||S|}{p} + p^{\frac{2-\epsilon}{4}}|A|^{\frac{1}{2}}|B|^{\frac{1}{2}}|S|^{\frac{3}{4}}.\] Moreover,  if any line passing through the origin contains at most $k$ points from $B$, we have
    \[ I(P,S) \ll \frac{|A|^{\frac{1}{2}}|B||S|}{p} + k^{\frac{1}{4}}p^{\frac{1-\epsilon}{4}}|A|^{\frac{1}{2}}|B|^{\frac{1}{2}}|S|^{\frac{3}{4}}.\]
\end{theorem}

The following example shows that conditions on $S$ are necessary to obtain non-trivial improvements.

Let $\ell_1$ and $\ell_2$ be two lines passing through the origin. Let $A, B$ be the set of points on $\ell_1, \ell_2$, respectively. Let $S$ be the set of all $\theta\in SL_2(\mathbb{F}_p)$ such that $\theta(\ell_2)=\ell_1$, and $H$ be the group of matrices $\theta\in SL_2(\mathbb{F}_p)$ such that $\theta(\ell_2)=\ell_2$. Then 
\[ I(P,S) \sim p|A||B| \sim p^3  \sim \frac{|A|^{\frac{1}{2}}|B||S|}{p} + p^{\frac{2}{4}}|A|^{\frac{1}{2}}|B|^{\frac{1}{2}}|S|^{\frac{3}{4}}   . \]
Moreover, it is not hard to check that $S=gH$ for some $g\in \SL_2(\Fbb_p)$.



For small sets, we first look at an example, which says that $I(P, S)$ could be $|P||S|$. Let $S$ be a subset of matrices $\theta$ in $SL_2(\mathbb{F}_p)$ such that $\theta(0, 1)=(1, 0)$. So the size of $S$ can be arbitrary smaller than $p$. Let $P=\{ \lambda(e_1, e_2)\colon \lambda\in \mathbb{F}_p^*\}$, where $e_1=(1, 0), e_2=(0, 1)$. Then the size of $P$ can be arbitrary smaller than $p$ by choosing $\lambda$. With these sets $P$ and $S$, we have $I(P, S)=|P||S|$.

When we know better about structures of $B$ and $S$, then the following theorem is attained. 

\begin{theorem}\label{incidence-energy'}
For $\gamma\in (0, 1)$, there exists $\epsilon=\epsilon(\gamma)>0$ such that the following holds.

    Let $p$ be a sufficiently large prime. Let $P =A \times B \subseteq \l( \Fbb_p^2 \times \Fbb_p^2 \r) \setminus \lo (0, 0, 0, 0)\ro $, and let $S$ be a symmetric subset of $\SL_2 \l( \Fbb_p \r)$ such that $p^\gamma < \lv S\rv < p^{3-2\gamma} $ and $\lv S\cap gH \rv < p^{\frac{-\gamma}{2}}\lv S\rv$ for any subgroup $H \subsetneq \SL_2 \l( \Fbb_p \r)$ and $g \in \SL_2 \l( \Fbb_p \r)$. 
    \begin{enumerate}
        \item Assume $|B|\le p$, any line passing through the origin contains at most $k_1$ points from $B$, and $B$ determines at most $k_2$ distinct directions through the origin, then
        \[ I(P,S) \lesssim k_1^{\frac{1}{2}}|A|^{\frac{1}{2}}|S|+  \frac{k_1^{\frac{1}{2}}\lv B\rv^{\frac{1}{2}}|A|^{\frac{1}{2}}|S|^{\frac{3}{4}}}{p^{\frac{\epsilon}{4}}} + \frac{k_1^{\frac{1}{8}}\lv B\rv^{\frac{3}{4}}|A|^{\frac{1}{2}} |S|^{\frac{3}{4}}}{p^{\frac{\epsilon}{4}}}+\frac{k_1^{\frac{1}{4}}k_2^{\frac{1}{4}}\lv B\rv^{\frac{1}{2}} |A|^{\frac{1}{2}}|S|^{\frac{3}{4}}}{p^{\frac{\epsilon}{4}}}.\]
        \item Assume $|B|\le p^{8/15}$, and any line passing through the origin contains at most $k_1$ points from $B$, then
        \[ I(P,S) \lesssim  k_1^{\frac{1}{2}}|A|^{\frac{1}{2}}|S|+  \frac{k_1^{\frac{1}{2}}\lv B\rv^{\frac{1}{2}}|A|^{\frac{1}{2}}|S|^{\frac{3}{4}}}{p^{\frac{\epsilon}{4}}}+\frac{k_1^{\frac{1}{15}}|B|^{\frac{187}{225}} |A|^{\frac{1}{2}} |S|^{\frac{3}{4}}}{p^{\frac{\epsilon}{4}}}.\]
    \end{enumerate}
\end{theorem}
We note that the bound of $I(P, S)$ in this theorem is smaller than $|P||S|$ provided that $|A|\gg 1$.

\begin{remark}\label{rm24}
    If we remove the factor $p^{\epsilon/4}$ in the incidence bounds of Theorem \ref{incidence-energy'}, then the conditions on the set $S$ are not required.
\end{remark}



\subsection{Incidence mixing in the $SL_2(\mathbb{F}_p)$ action graph (Theorem \ref{incidence})}

Theorem \ref{incidence} will be proved by using tools from discrete Fourier analysis. We first recall some basic notations.

For $n\in \Nbb$, let $f:\Fbb_p^n \rar \Cbb$ be a complex valued function. The Fourier transform of $f$, denoted by $\widehat{f}$, is defined by
\[\widehat{f}(m):=p^{-n}\sum_{x\in\Fbb_p^n} \chi (-m\cdot x)f(x),\]
where $\chi$ is a nontrivial additive character of $\Fbb_p$. We
have the following basic properties of $\widehat{f}$.
\begin{itemize}
    \item The orthogonality property:
\[ \sum_{\alpha \in \Fbb_p^n}\chi (\beta \cdot \alpha) =\begin{cases}
    0 ,& \text{if}\quad \beta \ne (0,\ldots ,0) ,\\
    p^n ,& \text{if}\quad \beta = (0,\ldots ,0) .
\end{cases}\]
\item The Fourier inversion formula:
    \[ f(x)= \sum_{m\in \Fbb_p^n}\chi (m\cdot x) \widehat{f}(m). \]
\item The Plancherel formula:
    \[ \sum_{m\in \Fbb_p^n} \lv \widehat{f} (m)\rv^2 =p^{-n} \sum_{x\in \Fbb_p^n} \lv f(x) \rv^2. \]
\end{itemize}

For $A\subseteq \Fbb_p^d $, by abuse of notation, we also denote its characteristic function by $A(x)$, i.e. $A(x)=1$ if $x\in A$ and $A(x)=0$ if $x\not\in A$.  


To proceed further, we need three lemmas. For each $x =\l( x_1,x_2 \r) \in \Fbb_p^2$, we define $x^\perp = \l( -x_2,x_1 \r)$. Note that $x\cdot y^\perp$ measures the area of the triangle with three vertices $x$, $y$, and the origin.  The first lemma presents a fact that the group $SL_2(\mathbb{F}_p)$ preserves areas of triangles with one vertex pinned at the origin.
\begin{lemma}\label{lmmay2}
Let $x, y, u, v$ be points in $\mathbb{F}_p^2\setminus\{(0, 0)\}$. If there exists $\theta \in \SL_2(\mathbb{F}_p)$ such that $\theta x=u$ and $\theta y=v$, then $x\cdot y^\perp=u\cdot v^\perp$. In the inverse direction, if $x$ and $y$ are not on the same line passing through the origin and $x\cdot y^\perp=u\cdot v^\perp$, then there exists unique $\theta\in SL_2(\mathbb{F}_p)$ such that $\theta x=u$ and $\theta y=v$. 
\end{lemma}
\begin{proof}
    Assume there exists $\theta \in \SL_2(\mathbb{F}_p)$ such that $\theta x=u$ and $\theta y=v$. Writing $\theta$ in the form 
    \[ \theta = \begin{bmatrix}
        a & b \\
        c & d
    \end{bmatrix} ,\]
    where $ad-bc=1$, $x=(x_1 ,x_2)$, and $y= (y_1 ,y_2)$. We have $u=(ax_1 +bx_2,cx_1 +dx_2) ,v = (ay_1 +by_2,cy_1 +dy_2) $. Then, 
    \begin{align*}
        u\cdot v^\perp &= -(ax_1+bx_2)(cy_1+dy_2) + (cx_1+dx_2)(ay_1+by_2) \\
        &= -(ad-bc)x_1y_2 +(ad-bc)x_2y_1 = x\cdot y^\perp .
    \end{align*}

In the inverse direction, since $x \ne ky$ for all $k \in \Fbb_p$, we have $x\cdot y^\perp \ne 0$. Indeed, writing $x$ as $x=\l( x_1,x_2 \r) \ne (0, 0)$. Since $x \ne (0, 0)$, we can assume that $x_1 \ne 0$. Therefore, if $y=\l( y_1,y_2 \r)$ satisfies $x\cdot y^\perp =0$, then $x$ and $y$ belong to a line passing through the origin, a contradiction. Similarly, we obtain $u\cdot v^\perp\ne 0$. Let $\theta$ be the matrix that maps the basis  $\lo x,y\ro$ to the basis $\lo u,v\ro$. We show that $\theta \in \SL_2(\mathbb{F}_p)$. Indeed, we write $\theta = \begin{bmatrix}
        a & b\\ 
        c & d
    \end{bmatrix}$ and $x= (x_1 ,x_2)$, and $y=(y_1 ,y_2)$. This implies $u=(ax_1 +bx_2,cx_1 +dx_2)$ and $v = (ay_1 +by_2,cy_1 +dy_2) $. Therefore, $x \cdot y^\perp = u\cdot v^\perp = (ad-bc)x\cdot y^\perp$, so $ad-bc=1$. In other words, $\theta\in SL_2(\mathbb{F}_p)$. 
\end{proof}

The next lemma tells us that the action of $SL_2(\mathbb{F}_p)$ on the plane $\mathbb{F}_p^2$ is transitive with multiplicity of $\sim p$. We give a general proof for the case $SL_n(\mathbb{F}_p)$. 
\begin{lemma}\label{lmmay1}
    For any $m,m'\in \Fbb_p^2 \setminus \lo (0,0)\ro $, define
    \[\M_{2,p}\l( m, m' \r) := \lo T\in \SL_2 \l( \Fbb_p \r) \colon Tm=m' \ro . \]
   Then  $|\M_{2,p}\l( m, m' \r)|= p.$
\end{lemma}

\begin{proof}
It follows from \cite[Theorem 13.3.3]{M.P.13} that $\SL_2(\mathbb{F}_p)$ is a group of size
    \[\lv \SL_2  \l( \Fbb_p \r) \rv = p\l( p^2-1 \r) .\]
Considering the group action of $\SL_2 \l( \Fbb_p \r) $ on $\Fbb_p^2$, $A : x \mapsto Ax , \forall A \in \SL_2 \l( \Fbb_p \r) ,x\in \Fbb_p^2$. For $m \ne (0,0)$, let $\text{Orb}(m)$ be the orbit of $m$, i.e. $\text{Orb}(m)=\{Tm\colon T\in \SL_2(\Fbb_p)\}$. Since $m\ne (0,0)$, there exists $m_1' \in \Fbb_p^2 \setminus \{ (0,0) \}$ such that 
$\lo m, m_1'\ro$ forms a linear independent system. So, let  
\[ T_m' = \begin{bmatrix}
    m&  m_1'
\end{bmatrix},\]
we have $\det T_m' = \lambda_m \ne 0$. Let $m_1= \lambda_m^{-1}m_1'$,  and
\[ T_m = \begin{bmatrix}
    m&  m_1
\end{bmatrix}.\]
Then,  $\det T_m = \lambda_m \cdot \lambda_m^{-1}=1$, so $T_m \in \SL_2(\Fbb_p)$. Now, for all $m' \in \Fbb_p^2\setminus \lo (0,0) \ro$, we observe that $T_{m'} \circ (T_m)^{-1}\in \SL_2(\Fbb_p)$, and $T_{m'} \circ (T_m)^{-1}(m)=m'$ so $\M_{2,p}(m,m') \ne \emptyset , \forall \, m,m' \ne (0,0)$. 

For $m, m'\in \mathbb{F}_p^2\setminus \{(0,0)\}$, let $T$ be an element in $\M_{2,p}(m, m')$. Then for all $T'\in \M_{2,p}(m, m')$, there exists $A\in \text{Stab}(m)$ such that $TA=T'$. This implies that $|\M_{2,p}\l( m,m' \r)|=\lv \text{Stab}(m) \rv $ for all $m,m'\in \Fbb_p^2 \setminus \{(0,0)\} $.

Moreover, for any $m\ne (0,0)$, there is no $T \in \SL_2(\Fbb_p)$ such that $T m =(0,0)$. Hence, we have $\text{Orb}(m) = \Fbb_p^2 \setminus \lo (0,0)\ro$. Thus, 
by the Orbit-Stabilizer theorem, 
\[ \lv \text{Stab}(m) \rv =\frac{\lv \SL_2 \l( \Fbb_p \r) \rv}{\lv \text{Orb} (m)\rv} = \frac{p(p^2-1)}{p^2-1}=p .\]This completes the proof. 
\end{proof}

The next lemma is an $L^2$ bound for the dot-product function.

\begin{lemma}[\cite{P.V.17}]\label{lmmay3}
    Let $A$ and $B$ be subsets of $\, \mathbb{F}_p^2$. The number of tuples $(x_1, x_2, y_1, y_2)\in A\times A\times B\times B$ such that $x_1\cdot x_2^{\perp}=y_1\cdot y_2^{\perp}$ is at most 
    \begin{equation}\label{eq111}\frac{|A|^2|B|^2}{p}+Cp^2|A||B|,\end{equation}
    for some positive constant $C$. 
\end{lemma}
When $P\subset \mathbb{F}_p^4$ is a general set, the above upper bound can be replaced by $\frac{|P|^2}{p}+Cp^2|P|$. Let $k_A$ and  $k_B$ be the maximal number of points from $A$ and $B$ on a line passing through the origin, respectively. Assume that $\min\{k_A, k_B\}=k$, then, with the same argument, the bound (\ref{eq111}) can be replaced by
    \[\frac{|A|^2|B|^2}{p}+Cpk|A||B|.\]


With these three lemmas in hand, we are ready to prove Theorem \ref{incidence}.

\begin{proof}[Proof of Theorem \ref{incidence}]
Without loss of generality, we assume that $(0, 0, 0, 0)\not\in P$, since this element only contributes $|S|$ incidences to the incidence bound. 

    By using the Fourier transformation and the Fourier inversion formula, we have 
    \begin{align*}
        I(P,S) & = \sum_{\substack{p=(x,y )\in P \\ \theta \in S}} 1_{\theta y=x} =\frac{1}{p^2} \sum_{m\in \Fbb_p^2} \sum_{\substack{(x,y)\in P, \\ \theta \in S}} \chi \l( m\cdot \l( x-\theta y \r) \r) \\
        & = \frac{\lv P\rv \lv S\rv}{p^2} + \frac{1}{p^2} \sum_{ m\in \Fbb_p^2 \setminus \lo (0,0)\ro } \sum_{\substack{ (x,y)\in P ,\\ \theta \in S }} \chi \l( m \cdot \l( x -\theta y \r) \r) \\
        & = \frac{\lv P\rv \lv S\rv}{p^2} + p^2\sum_{m\ne (0, 0)} \sum_{\theta \in S} \widehat{P} \l( -m,\theta^t m\r).
    \end{align*}
By using the Cauchy--Schwarz inequality, one has 
\begin{align*}
    &\sum_{\theta\in S}\sum_{m\ne (0, 0)}\widehat{P}(-m, \theta^Tm)\le |S|^{\frac{1}{2}}\left(\sum_{\theta\in \SL_2 \l( \Fbb_p \r)}\sum_{m_1, m_2\ne (0, 0)}\widehat{P}(-m_1, \theta^tm_1)\overline{\widehat{P}(-m_2, \theta^tm_2)}\right)^{\frac{1}{2}}.
\end{align*}
We now observe 
\begin{align*}
    &\sum_{\theta\in \SL_2 \l( \Fbb_p \r)}\sum_{m_1, m_2\ne (0, 0)}\widehat{P}(-m_1, \theta^tm_1)\overline{\widehat{P}(-m_2, \theta^tm_2)}\\&=\frac{1}{p^8}\sum_{\theta}\sum_{m_1, m_2\ne (0, 0)}\sum_{(x_1, y_1), (x_2, y_2)}P(x_1, y_1)P(x_2, y_2)\chi(-m_1x_1+\theta^tm_1y_1)\chi(m_2x_2-\theta^tm_2y_2)\\
    &=\sum_{m_1, m_2}-\sum_{m_1=(0, 0), m_2\ne (0, 0)}-\sum_{m_1\ne (0, 0), m_2=(0, 0)}-\sum_{m_1=m_2=(0, 0)}\\
    &=:I-II-III-IV.
\end{align*}
We now estimate each term separately. 
\begin{align*}
    &I=\frac{1}{p^8}\sum_{\theta}\sum_{m_1, m_2\in \mathbb{F}_p^2}\sum_{(x_1, y_1), (x_2, y_2)}P(x_1, y_1)P(x_2, y_2)\chi(m_1(\theta y_1-x_1))\chi(m_2(\theta y_2-x_2))\\
    &=\frac{1}{p^4}\sum_{\theta}\sum_{(x_1, y_1), (x_2, y_2)}P(x_1, y_1)P(x_2, y_2)1_{\theta y_1=x_1}1_{\theta y_2=x_2}.\\
\end{align*}
By using  Lemma \ref{lmmay2}, one has
\begin{align*}
    &I=\frac{1}{p^4}\sum_{\theta}\sum_{\lambda\in \mathbb{F}_p}\sum_{\substack{x_1, y_1, x_2, y_2\\~y_1=\lambda y_2, x_1=\lambda x_2}}A(x_1)A(x_2)B(y_1)B(y_2)1_{\theta y_1=x_1}1_{\theta y_2=x_2}\\
    &+\frac{1}{p^4}\sum_{x_1, y_1, x_2, y_2}A(x_1)A(x_2)B(y_1)B(y_2)1_{y_1\cdot y_2^\perp=x_1\cdot x_2^\perp}.
\end{align*}
Notice that Lemma \ref{lmmay1} and Lemma \ref{lmmay2} tell us that the first sum can be bound by at most $p^2|A||B|/p^4$.

By Lemma \ref{lmmay3}, the second sum can be at most 
\[\frac{|A|^2|B|^2}{p^5}+C\frac{|A||B|}{p^2}.\]

In other words, we have 
\[I\le \frac{|A|^2|B|^2}{p^5}+(C+1)\frac{|A||B|}{p^2}.\]

Moreover,
\[IV=\frac{(p^3-p)|A|^2|B|^2}{p^8}=\frac{|A|^2|B|^2}{p^5}-\frac{|A|^2|B|^2}{p^7}.\]

Thus, $I-IV\ll \frac{|A||B|}{p^2}$.
Regarding $II$, 
\begin{align*}
    II & =\frac{1}{p^8}\sum_\theta \sum_{( x_1,y_1)} P(x_1,y_1) \l( \sum_{(x_2,y_2)} P(x_2,y_2) \l( \sum_{m_2\ne 0} \chi \l( m_2 \l( x_2 -\theta y_2 \r) \r) \r) \r) \\
    & = \frac{1}{p^8} \sum_\theta \sum_{(x_1,y_1)} P(x_1 ,y_1) \l( \sum_{\substack{(x_2,y_2), \\ x_2 =\theta y_2}} P(x_2,y_2) \l( p^2-1 \r)  - \sum_{ \substack{(x_2,y_2), \\ x_2 \ne \theta y_2 }} P(x_2,y_2)  \r) \\
    & = \frac{1}{p^8} \lv P\rv \l( \l( p^2-1\r)\sum_{\substack{(x_2,y_2) \in P,\\ \theta \in \SL_2 \l( \Fbb_p \r) ,\\ x_2 =\theta y_2}} 1  - \sum_{ \substack{(x_2,y_2) \in P,\\ \theta \in \SL_2 \l( \Fbb_p \r) ,\\ x_2 \ne \theta y_2 }} 1\r) \\
    & = \frac{1}{p^8} \lv P\rv \l( \lv A\rv \lv B\rv p\l( p^2 -1\r) - \lv A\rv \lv B\rv \l( p^3 -2p \r) \r) = \frac{\lv P\rv^2}{p^7}.
\end{align*}

Similarly, we obtain  $III =\frac{\lv P\rv^2}{p^7} .$
Putting all estimates together, we conclude that 
\[\left\vert I(P, S)-\frac{|P||S|}{p^2}\right\vert\ll p\sqrt{|P||S|}+|S|.\]
This completes the proof.
\end{proof}

\subsection{Energy refinement and subgroup non-concentration (Theorem \ref{thm: 4.12})}\label{sectionenergy}


For a set $S\subset SL_2(\mathbb{F}_p)$, we define the energy $\Ebf(S, S)$ by
\[\Ebf(S, S):=\#\{(a, b, c, d)\in S^4\colon ab=cd\}.\]
The trivial bound of $\Ebf(S, S)$ is $|S|^3$. When $S$ is a large set, Babai, Nikolov, and Pyber proved in \cite{B.N.P.08} that
 \begin{equation}\label{eq-energy1} \Ebf \l( S, S \r) \ll p^2 |S|^2 + \frac{|S|^4}{p^3} .\end{equation}
This bound is sharp. To see its sharpness, we provide an example here.

  Let $S$ be the set of matrices of the form
    \[ \begin{bmatrix}
        \ast & -x \\
        x^{-1} & 0
    \end{bmatrix}, \]
    where $x,\ast \in \Fbb_p \setminus \{ 0 \}$. Then, $S$ is a subset of $\SL_2 \l( \Fbb_p \r)$ and $\lv S\rv =(p-1)^2$. We consider following equation
    \begin{align}\label{eq: 3}
        \begin{bmatrix}
        \ast_1 & -x \\
        x^{-1} & 0
    \end{bmatrix} \begin{bmatrix}
        \ast_2 & -y \\
        y^{-1} & 0
    \end{bmatrix} =  \begin{bmatrix}
        \ast_1' & -x' \\
        (x')^{-1} & 0
    \end{bmatrix} 
    \begin{bmatrix}
        \ast_2' & -y' \\
        (y')^{-1} & 0
    \end{bmatrix},
    \end{align} 
    where all matrices are in $S$. The equation is equivalent to 
    \[ \begin{cases}
        \ast_1 \ast_2 -xy^{-1} = \ast_1' \ast_2 ' -x'(y')^{-1}, \\
        -y\ast_1 = -y' \ast_1' ,\\
        \ast_2 x^{-1} =\ast_2'(x')^{-1} ,\\
        -yx^{-1} =-y' x^{-1}.
    \end{cases} \]
    This implies $\frac{y}{y'}=\frac{x}{x'}=\frac{\ast_2}{\ast_2'}=\frac{\ast_1'}{\ast_1}$. Therefore, for fixed $\ast_1 , \ast_1' ,x,x'$, there exist $(p-1)^2$ tuples $\l( y,y',\ast_2,\ast_2' \r) \in \l( \Fbb_p \setminus \{ 0\} \r)^4$ such that the equation (\ref{eq: 3}) holds. In other words, for each pair of matrices $(A,C) \in S^2$, there exist $(p-1)^2$ pairs of matrices $(B,D) \in S^2$ such that $AB=CD$. Hence,
    \[ \Ebf (S,S) \gg \lv S\rv^2 (p-1)^2 =(p-1)^6 \sim p^2 (p-1)^4 + \frac{(p-1)^8}{p^3} = p^2\lv S\rv^2 + \frac{\lv S\rv^4}{p^3}. \]
When the set $S$ is of small size, one would hope to have an upper bound of $\Ebf(S, S)$ that does not depend on $p$. In this paper, we make use of the following result due to Bourgain and Gamburd in \cite{B.A.08} to derive such a bound over prime fields.
\begin{theorem}[Proposition $2$, \cite{B.A.08}]\label{thm: 7.1}
Let $p$ be a sufficiently large prime, and let $\eta$ be a symmetric probability measure on $\SL_2(\Fbb_p)$ and $0< \gamma < \frac{3}{4}$, such that
\begin{itemize}
\item[(1)] $\lV \eta  \rV_{\infty} < p^{-\gamma} ;$
\item[(2)] $\eta(gH)<p^{\frac{-\gamma }{2}} $ for any proper subgroup $H\subset SL_2( \Fbb_p ), g\in SL_2( \Fbb_p );$
\item[(3)] $\Vert \eta\Vert_2> p^{\frac{-3}{2} +\gamma}$.
\end{itemize}
Then there exists $\epsilon = \epsilon \l( \gamma \r) >0$ such that
$$
\Vert\eta*\eta\Vert_2 < p^{-\epsilon} \Vert\eta\Vert_2.
$$
\end{theorem}

Moreover, following the proof of Theorem \ref{thm: 7.1}, we have $\epsilon < \gamma$. 

For $0< \gamma <\frac{3}{4}$, let $S$ be a symmetric subset of $\SL_2 \l( \Fbb_p \r)$ such that $p^{\gamma}<|S|<p^{3-2\gamma}$ and for any proper subgroup $H \subsetneq \SL_2 \l( \Fbb_p \r) $, $g \in \SL_2 \l( \Fbb_p \r)$ we have $\lv S\cap gH\rv < p^{\frac{-\gamma }{2}}\lv S\rv $. Let $\mu_S : \SL_2 \l( \Fbb_p \r) \rar \Rbb $ be the function defined by $\mu_S \l( g \r) =\frac{1}{\lv S\rv}$ if $g\in S$ and $\mu (g)=0$ if $g \notin S$. Then, $\mu_S$ satisfies all conditions of Theorem \ref{thm: 7.1}. Indeed, we have 
\begin{itemize}
    \item[(1)] $\lV \mu_S \rV_\infty = \max_{g} \mu_S(g) = \frac{1}{|S|} < p^{-\gamma }$,
    \item[(2)] $\mu_S (gH) = \frac{|S\cap gH|}{|S|} < p^{\frac{-\gamma }{2}}$, for any proper subgroup $H \subset \SL_2(\Fbb_p)$ and $g\in \SL_2(\Fbb_p)$,
    \item[(3)] $\lV \mu_S \rV_2 = \l( \sum_{g\in \SL_2(\Fbb_p)} \mu_S(g)^2 \r)^{\frac{1}{2}} = \frac{1}{|S|^{\frac{1}{2}}} > p^{\frac{-3}{2}+\gamma}.$
\end{itemize}
Therefore,
\[ \lV \mu_S \ast \mu_S \rV_2 < p^{-\epsilon} \lV \mu_S \rV_2 = p^{-\epsilon} \frac{1}{\lv S\rv^{\frac{1}{2}}}. \]

One the other hand, $\mu_S \ast \mu_S (g) = \frac{\lv  \lo (a,b)\in S \times S \colon ab=g \ro \rv}{\lv S\rv^2} $. Then,
\[ \lV \mu_S \ast \mu_S \rV_2^2 = \frac{\Ebf \l( S,S \r)}{\lv S\rv^4}.\]
In other words, we have proved the following corollary. 
\begin{corollary}\label{co:36}
For any $\gamma\in (0, 1)$, there exists $\epsilon=\epsilon(\gamma)>0$ such that the following holds.
 Let $p$ be a sufficiently large prime, let $S$ be a symmetric subset of $\SL_2 (\Fbb_p )$ such that $p^\gamma < |S| < p^{3-2\gamma } $ and $ |S \cap gH| < p^{\frac{-\gamma}{2}}|S|$ for any subgroup $H \subsetneq \SL_2 (\Fbb_p) $ and $g\in \SL_2(\Fbb_p)$. Then, we  have
\begin{equation}
\Ebf \l( S,S \r) < \frac{\lv S\rv^3}{p^{\epsilon}}.
\end{equation}    
\end{corollary}
\begin{proof}
    For any $\gamma\in (0, \frac{3}{4})$, it follows directly from the previous computation. For any $\gamma\in [3/4, 1)$, we observe that 
    \[(p^\gamma, p^{3-2\gamma})\subset (p^{\gamma'}, p^{3-2\gamma'}), ~~p^{-\frac{\gamma}{2}}|S|\le p^{-\frac{\gamma'}{2}}|S|,\]
    for any $\gamma'\in (0, 3/4)$. Then, by choosing $\epsilon=\epsilon(\gamma')$ for any $\gamma'\in (0, 3/4)$, the corollary follows.
\end{proof}

 Let $A \times B \subseteq \Fbb_p^2 \times \Fbb_p^2 $ and $S\subseteq \SL_2 \l( \Fbb_p \r)$. Then,
    \[ I(A\times B,S) \le N^{\frac{1}{2}}\lv A\rv^{\frac{1}{2}} ,\]
    where $N$ is the number of $\l( b,b',\theta ,\theta' \r) \in B \times B \times S \times S$ such that $\theta b  =\theta 'b'$.

Indeed,  for each $\l( a,b \r) \in A\times B$, denote $s_{(a,b)} $ as the number of $\theta \in S$ such that $\theta b=a$. Therefore,
    \begin{align*}
        I(P,S) & = \sum_{a\in A} \l( \sum_{b\in B} s_{(a,b)} \r)  \le  \lv A\rv^{\frac{1}{2}} \l( \sum_{a\in A} \l( \sum_{b\in B} s_{(a,b)} \r)^2 \r)^{\frac{1}{2}} \\
        & =  \lv A\rv^\frac{1}{2} \l( \sum_{a\in A} \lv \lo \l( b,b',\theta ,\theta' \r) \in B\times B \times S \times S \colon \theta b=\theta 'b' =a \ro \rv \r)^{\frac{1}{2}} \\
        & \le \lv A\rv^{\frac{1}{2}} \cdot N^{\frac{1}{2}}.
    \end{align*}
To bound $N$, we observe that the equation $\theta b=\theta'b'$ gives $(\theta')^{-1}\theta b=b'$. So, $N$ can be viewed as the number of incidences between $B\times B$ and the multi-set $S^{-1}S$.  If $(0, 0) \notin B$, by following the proof of Theorem \ref{incidence} identically, we obtain
\[N\ll  \frac{\lv B\rv^2 \lv S\rv^2}{p^2} +p\lv B\rv (\Ebf (S,S))^{\frac{1}{2}},\]
and if any line passing through the origin contains at most $k$ points from $B$, then
\[ N \ll  \frac{\lv B\rv^2 \lv S\rv^2}{p^2} +p^{\frac{1}{2}}k^{\frac{1}{2}}\lv B\rv (\Ebf (S,S))^{\frac{1}{2}}.\]

As mentioned in (\ref{eq-energy1}), 
 \[ \Ebf \l( S, S \r) \ll p^2 |S|^2 + \frac{|S|^4}{p^3} .\]
Substituting this bound into $I(A\times B, S)$ implies 
 \[ I(A\times B ,S)\ll \frac{|A|^{\frac{1}{2}}|B||S|}{p}+\frac{|A|^{\frac{1}{2}}|B|^{\frac{1}{2}}|S|}{p^{\frac{1}{4}}}+p|A|^{\frac{1}{2}}|S|^{\frac{1}{2}}|B|^{\frac{1}{2}}.\]

Compared to the bound of Theorem \ref{incidence}, this result is weaker. 

However, if we use Corollary \ref{co:36} instead, then
    \[ I(P,S) \ll \frac{|A|^{\frac{1}{2}}|B||S|}{p} + p^{\frac{2-\epsilon}{4}}|A|^{\frac{1}{2}}|B|^{\frac{1}{2}}|S|^{\frac{3}{4}}.\]
Moreover,  if any line passing through the origin contains at most $k$ points from $B$, then
    \[ I(P,S) \ll \frac{|A|^{\frac{1}{2}}|B||S|}{p} + k^{\frac{1}{4}}p^{\frac{1-\epsilon}{4}}|A|^{\frac{1}{2}}|B|^{\frac{1}{2}}|S|^{\frac{3}{4}}.\]
This completes the proof of Theorem \ref{thm: 4.12}.

\subsection{Two alternative approaches yield weaker bounds}
This section presents two different approaches without techniques from Fourier analysis. Although the resulting bounds are weaker compared to those of Theorem \ref{incidence} and Theorem \ref{thm: 4.12}, the methods will be useful for us when studying the case of small sets. 

\begin{theorem}\label{thm:3.7}
For any $\gamma\in (0, 1)$, there exists $\epsilon=\epsilon(\gamma)>0$ such that the following holds.
    Let $p$ be a sufficiently large prime. Let $P=A\times B \subseteq (\Fbb_p^2 \times \Fbb_p^2 ) \setminus \lo (0, 0, 0, 0)\ro$, and let $S$ be a symmetric subset of $\SL_2 (\Fbb_p )$ such that $p^\gamma < |S| < p^{3-2\gamma } $ and $ |S \cap gH| < p^{\frac{-\gamma }{2}}|S|$ for any subgroup $H \subsetneq \SL_2 (\Fbb_p) $ and $g\in \SL_2(\Fbb_p)$. Assume any line passing through the origin contains at most $k$ points from $B$. Then, 
    \begin{itemize}
        \item[(1)] if $|B|< k^{\frac{1}{2}}p$, we have 
        \[ I(P,S) \lesssim k^{\frac{1}{2}}|A|^{\frac{1}{2}} |S| + k^{\frac{1}{4}} p^{\frac{1-\epsilon}{4}} |A|^{\frac{1}{2}} |B|^{\frac{1}{2}} |S|^{\frac{3}{4}} ,\]
        \item[(2)] if $|B|\ge k^{\frac{1}{2}}p$, we have 
        \[ I(P,S) \lesssim 
        \frac{|A|^{\frac{1}{2}}|B||S|^{\frac{3}{4}}}{p^{\frac{1+ \epsilon }{4}}}.\]
    \end{itemize}
   
\end{theorem}
Compared to the
bound of Theorem \ref{thm: 4.12}, which is
\[ I(P,S) \ll \frac{|A|^{\frac{1}{2}}|B||S|}{p} + k^{\frac{1}{4}}p^{\frac{1-\epsilon}{4}}|A|^{\frac{1}{2}}|B|^{\frac{1}{2}}|S|^{\frac{3}{4}},\]
we can see that
\begin{itemize}
    \item if $|B|< k^{\frac{1}{2}}p $, then
    \begin{align*}
    \frac{|A|^\frac{1}{2}|B||S|}{p}< k^{\frac{1}{2}}|A|^{\frac{1}{2}}|S|,
\end{align*}
    \item if $|B|\ge k^{\frac{1}{2}}p $, then
    \begin{align*}
    \frac{|A|^\frac{1}{2}|B||S|}{p}, k^{\frac{1}{4}}p^{\frac{1-\epsilon}{4}}|A|^{\frac{1}{2}}|B|^\frac{1}{2}|S|^{\frac{3}{4}} \le \frac{|A|^{\frac{1}{2}}|B||S|^{\frac{3}{4}} }{p^{\frac{1+\epsilon}{4}}},
\end{align*}
since $|S|< p^{3-2\gamma} < p^{3-\epsilon}$. In other words, Theorem \ref{thm: 4.12} is better than Theorem \ref{thm:3.7}.
\end{itemize}

\begin{proof}[Proof of Theorem \ref{thm:3.7}]
Since the identity matrix $I$ contributes at most $$\min \lo |A|,|B|\ro \ll k^{\frac{1}{4}}p^{\frac{1-\epsilon }{4}}\lv A\rv^{\frac{1}{2}} \lv B\rv^{\frac{1}{2}} |S|^{\frac{3}{4}}$$ incidences to $I(P,S)$, we may assume without loss generality that $I \not\in S$.

    For a set $D\subset \mathbb{F}_p^2\times \mathbb{F}_p^2$ and $\theta\in SL_2(\mathbb{F}_p)$, by $i_{D}(\theta)$, we mean the number of incidences between $\theta$ and the set $D$. 

We have 
\begin{align*}
I(P, S)=\sum_{\theta\in S}i_P(\theta) \le |A|^{\frac{1}{2}}\left(I(B\times B, S')\right)^{\frac{1}{2}},
\end{align*}
where $S'$ be the multi-set of elements of the form $a^{-1}b$, where $a,b\in S$. Set $E=B\times B$. For $\theta\in S'$, let $m(\theta)$ be the multiplicity of $\theta$ in $S'$, and for $(x, y)\in E$, let $n(x, y)$ be the number of elements $(u, v)\in E$ such that $(u, v)=\lambda (x, y)$ for some $\lambda\in \mathbb{F}^*_p$. 

We now bound the number of incidences between $S'$ and $B\times B$. By $\overline{S'}$, we mean the set of distinct elements in $S$. We have 
\[I(B\times B, S')=\sum_{\theta\in \overline{S'}}m(\theta)\sum_{(x, y)\in E}\theta(x, y),\]
where $\theta(x, y)=1$ if $\theta y=x$, and $0$ otherwise. 

We now write
\begin{align*}
    &I(B\times B, S')=\sum_{\theta\in \overline{S'}}m(\theta)\sum_{(x, y)\in E}\theta(x, y)\\
    &=\sum_{\theta\in \overline{S'}}m(\theta)\sum_{(x, y)\in E, n(x, y)=1}\theta(x, y)+\sum_{\theta\in \overline{S'}}m(\theta)\sum_{(x, y)\in E, n(x, y)>1}\theta(x, y)=I+II.
\end{align*}

Let $E'$ be the set of $(x, y)\in E$ such that $n(x, y)=1$. 
To bound $I$, 

\begin{align*}
    I=\sum_{\theta\in \overline{S'}, ~i_{E'}(\theta)>1}m(\theta)\sum_{(x, y)\in E, n(x, y)=1}\theta(x, y)+\sum_{\theta\in \overline{S'}, ~i_{E'}(\theta)=1}m(\theta)\sum_{(x, y)\in E, n(x, y)=1}\theta(x, y)=I_1+I_2.
\end{align*}
It is clear that $I_2\le |S'|=|S|^2$. By the Cauchy--Schwarz inequality and Lemma \ref{lmmay3}, we have 
\begin{align*}
     I_1 & =\sum_{\theta \in\overline{S'},i_{E'}(\theta)>1}m (\theta )i_{E'}(\theta) \le  \l(\sum_{\theta \in\overline{S'},i_{E'}(\theta)>1}m (\theta )^2 \r)^\frac{1}{2} \l( \sum_{\theta \in\overline{S'},i_{E'}(\theta)>1} i_{E'}(\theta)^2 \r)^{\frac{1}{2}}   \\
     & \ll \l( \Ebf (S,S) \r)^{\frac{1}{2}} \l( \sum_{\theta \in\overline{S'},i_{E'}>1}\binom{i_{E'}(\theta)}{2} \r)^{\frac{1}{2}} \\
     & \le \l( \Ebf (S,S) \r)^{\frac{1}{2}} \l( \lv \lo  (x_1,y_1,x_2,y_2) \in E \colon (x_1,y_1)\ne (x_2,y_2), \exists \theta \in \overline{S'}, \theta \text{ is incident to } (x_1,y_1) \text{ and } (x_2,y_2) \ro \rv \r)^{\frac{1}{2}} \\
     & \le \l( \Ebf (S,S) \r)^{\frac{1}{2}} \l( \lo \l( x_1,y_1,x_2,y_2 \r) \in B^4 \colon \exists \theta \in \overline{S'} , \theta x_1 =y_1 ,\theta x_2 =y_2 \ro \r)^{\frac{1}{2}} \\
     & \le \l( \Ebf (S,S) \r)^{\frac{1}{2}} \l( \lo \l( x_1,y_1,x_2,y_2 \r) \in B^4 \colon x_1 \cdot x_2^\perp =y_1 \cdot y_2^\perp \ro \r)^{\frac{1}{2}} \\
     &\ll (\Ebf (S,S))^{\frac{1}{2}}\cdot \l( \frac{|B|^2}{p^{\frac{1}{2}}} + k^\frac{1}{2}p^\frac{1}{2}|B| \r) .
    \end{align*}
    
Thus, using Corollary \ref{co:36}, we obtain 
\[I\le |S|^2+ \l( \frac{|B|^2}{p^{\frac{1}{2}}} + k^{\frac{1}{2}}p^{\frac{1}{2}}|B|\r) |S|^{\frac{3}{2}}p^{\frac{-\epsilon}{2}}.\]
We now consider $II$. 
\begin{align*}
    II &= \sum_{\theta\in \overline{S'}}m(\theta)\sum_{(x, y)\in E, n(x, y)>1}\theta(x, y)=\sum_{i=1}^{\lf \log_2 (k) \rf} \sum_{\theta\in \overline{S'}}m(\theta)\sum_{(x, y)\in E, 2^i \le n(x, y) < 2^{i+1}}\theta(x, y).
\end{align*}
By using pigeonhole principle, there exists $\ell =2^{i_0} $ such that
\[ II \lesssim  \sum_{\theta\in \overline{S'}}m(\theta)\sum_{(x, y)\in E, \ell \le n(x, y)< 2\ell}\theta(x, y) . \]
We say two pairs $(x, y)$ and $(x', y')$ are in the same congruence class if there exists $\lambda \in \Fbb_p^\ast$ such that $(x,y) = \lambda (x',y')$. Define $E''$ to be the set of congruence classes $[(x, y)]$ in $E$ such that $\ell \le n(x, y)< 2\ell$. 

Then we have 
\[II \lesssim \ell\sum_{\theta\in \overline{S'}}m(\theta)i_{E''}(\theta)=II_1+II_2.\]
Here 
\[II_1=\ell\sum_{\theta\in \overline{S'}, ~i_{E''}(\theta)=1}m(\theta)i_{E''}(\theta),~~~II_2=\ell\sum_{\theta\in \overline{S'}, ~i_{E''}(\theta)\ge 2}m(\theta)i_{E''}(\theta).\]
As above, we have 
\[II_1\le \ell |S|^2,\]
and 
\begin{align*}
    II_2 & \le \l( (\Ebf (S,S)) \r)^{\frac{1}{2}} \ell \l( \sum_{\theta\in \overline{S'}, ~i_{E''}(\theta)\ge 2} i_{E''}(\theta)^2 \r) ^{\frac{1}{2}}  \ll \l( \Ebf (S,S) \r)^\frac{1}{2} \l( \sum_{\theta \in\overline{S'},i_{E''}(\theta )_>1} \ell^2 \binom{i_{E''}(\theta)}{2} \r)^{\frac{1}{2}} \\
    & \le \l( \Ebf (S,S) \r)^\frac{1}{2} \l( \ell^2 \lv \lo  (x_1,y_1,x_2,y_2) \in E''\times E''\colon \exists \theta \in \overline{S'} , \theta \text{ is incident to $(x_1,y_1)$ and $(x_2,y_2)$} \ro \rv \r)^{\frac{1}{2}} \\
    & \le \l( \Ebf (S,S) \r)^\frac{1}{2} \l( \lv \lo  (x_1,y_1,x_2,y_2) \in B\times B \times B \times B\colon \exists \theta \in \overline{S'} , \theta y_1=x_1 ,\theta y_2 =x_2 \ro \rv \r)^{\frac{1}{2}} \\
    & \le \l( \Ebf (S,S) \r)^\frac{1}{2} \l( \lv \lo \l( x_1,y_1,x_2,y_2\r) \in B\times B \times B \times B \colon x_1\cdot x_2^\perp =y_1 \cdot y_2^\perp \ro \rv \r)^{\frac{1}{2}} \\
    & \le (\Ebf (S,S))^{\frac{1}{2}} \cdot \l( \frac{|B|^2}{p^{\frac{1}{2}}} + k^{\frac{1}{2}}p^{\frac{1}{2}}|B| \r) .
\end{align*} 

Putting these bounds together implies 
\[II\lesssim \ell |S|^2+ (\Ebf (S,S))^{\frac{1}{2}} \l( \frac{|B|^2}{p^{\frac{1}{2}}} + k^{\frac{1}{2}}p^{\frac{1}{2}}|B| \r).\]
Notice that $\ell\le k$. So,
\[I(B\times B, S')\lesssim  k|S|^2+\l( \frac{|B|^2}{p^{\frac{1}{2}}} + k^{\frac{1}{2}}p^{\frac{1}{2}}|B| \r) |S|^{\frac{3}{2}}p^{\frac{-\epsilon}{2}},\]
and then
\[ I(P,S) \lesssim k^{\frac{1}{2}}|A|^{\frac{1}{2}} |S| + k^{\frac{1}{4}} p^{\frac{1-\epsilon}{4}} |A|^{\frac{1}{2}} |B|^{\frac{1}{2}} |S|^{\frac{3}{4}} + p^{\frac{-1-\epsilon}{4}}|A|^{\frac{1}{2}}|B||S|^{\frac{3}{4}} .\]

A direct computation implies that
\begin{enumerate}
    \item if $|B|< k^{\frac{1}{2}}p$, then
    \[ I(P,S) \lesssim k^{\frac{1}{2}}|A|^{\frac{1}{2}} |S| + k^{\frac{1}{4}} p^{\frac{1-\epsilon}{4}} |A|^{\frac{1}{2}} |B|^{\frac{1}{2}} |S|^{\frac{3}{4}} ;\]
    \item if $|B| \ge k^{\frac{1}{2}}p$, then
    \[ I(P,S) \lesssim  p^{\frac{-1-\epsilon}{4}}|A|^{\frac{1}{2}}|B||S|^{\frac{3}{4}}.\]
\end{enumerate}
This completes the proof.
\end{proof}

\begin{theorem}\label{thm: 3.8}
    Let $p$ be a prime and $P=A\times B\subseteq (\Fbb_p^2 \times \Fbb_p^2) \setminus \lo (0, 0, 0, 0)\ro $. Assume any line passing through the origin contains at most $k$ points from $B$. Then, we have
    \[ I(P,S) \lesssim   \frac{|A||B|\lv S\rv^{\frac{1}{2}}}{p^{\frac{1}{2}}} + k^{\frac{1}{2}}p^{\frac{1}{2}}|A|^{\frac{1}{2}}|B|^{\frac{1}{2}}\lv S\rv^{\frac{1}{2}} + k\lv S\rv .\]
    
\end{theorem}
Compared to the bound of Theorem \ref{incidence}, which is
\[ I(P,S) \ll \frac{|A||B||S|}{p^2} + k^{\frac{1}{2}}p^{\frac{1}{2}}|A|^{\frac{1}{2}}|B|^{\frac{1}{2}}|S|^{\frac{1}{2}},\]
we can see that
\[ \frac{|A||B||S|}{p^2} \le \frac{|A||B||S|^{\frac{1}{2}}}{p^{\frac{1}{2}}},\]
under $|S|\le p^3-p$.
Therefore, Theorem \ref{incidence} is better than Theorem \ref{thm: 3.8}.
\begin{proof}[Proof of Theorem \ref{thm: 3.8}]
     Since the identity matrix $I$ contributes at most $\min\{|A|, |B|\}$ incidences to $I(P,S)$, we may assume without loss generality that $I \not\in S$.
    For $\theta\in S$ and set $D \subset \Fbb_p^2 \times \Fbb_p^2$, by $i_D(\theta)$ we mean the number of pairs $(x, y)\in D$ such that $\theta y=x$. For $(x,y)\in P$, let $n(x,y)$ be the number of elements $(u,v) \in P$ such that $(u,v)=\lambda (x,y)$ for some $\lambda \in \Fbb_p^\ast$. Then  
\begin{align*}
    I(P, S)\le \sum_{\theta\in S}i_P(\theta) & = \sum_{\theta \in S}\sum_{(x,y) \in P , n(x,y)=1} \theta (x,y)+ \sum_{\theta \in S}\sum_{(x,y) \in P , n(x,y)>1} \theta (x,y)\\ 
    & =: I_1 +I_2
\end{align*} 
where $\theta (x,y)=1$ if $\theta y=x$, and $0$ ortherwise. Let $P'$ be the set of $(x, y)\in P$ such that $n(x, y)=1$. By the Cauchy--Schwarz inequality and Lemma \ref{lmmay3}
\begin{align*}
    I_1&=\sum_{\theta \in S } i_{P'}
    (\theta) \le |S|^{\frac{1}{2}}\left(\sum_{\theta \in S}i_{P'}(\theta)^2\right)^{\frac{1}{2}} \ll \lv S\rv^{\frac{1}{2}} \l( \sum_{\theta \in S} \binom{i_{P'}(\theta)}{2} +\lv S\rv \r)^{\frac{1}{2}} \\
    & \le \lv S\rv^{\frac{1}{2}} \l( \lv \lo \l( x_1,y_1,x_2,y_2 \r) \in A\times B\times A \times B \colon \exists \theta \in S , \theta y_1 =x_1 ,\theta y_2 =x_2 \ro \rv +\lv S\rv \r)^{\frac{1}{2}}\\
    & \le \lv S\rv^{\frac{1}{2}} \l( \lv \lo \l( x_1,y_1,x_2,y_2 \r) \in A\times B\times A \times B \colon x_1\cdot x_2^\perp =y_1\cdot y_2^\perp \ro \rv +\lv S\rv \r)^{\frac{1}{2}}\\
    & \ll \lv S\rv^{\frac{1}{2}} \l(  \frac{|A||B|}{p^{\frac{1}{2}}} + (pk|A||B|)^{\frac{1}{2}} + \lv S\rv^{\frac{1}{2}} \r) .
\end{align*}
We now consider $I_2$,
\begin{align*}
    I_2 &= \sum_{\theta \in S}\sum_{(x,y) \in P , n(x,y)>1} \theta (x,y)=\sum_{i=1}^{\lf \log_2 (k) \rf} \sum_{\theta\in S}\sum_{(x, y)\in P, 2^i \le n(x, y) < 2^{i+1}}\theta(x, y) .
\end{align*}
By using pigeonhole principle, there exists $\ell =2^{i_0} $ such that
\[ I_2 \lesssim \sum_{\theta\in S}\sum_{(x, y)\in P, \ell \le n(x, y)< 2\ell}\theta(x, y) . \]
Define $P''$ to be the set of representatives $(x, y)$ in $P$ such that $\ell \le n(x, y)< 2\ell$. Then, by the Cauchy--Schwarz inequality, Lemma \ref{lmmay3} and note that $k \ge \ell$ we have 
\begin{align*}
    I_2 & \lesssim \ell\sum_{\theta\in S}i_{P''}(\theta) \le \ell |S|^{\frac{1}{2}}\left(\sum_{\theta \in S}i_{P''}(\theta)^2\right)^{\frac{1}{2}}  \ll \lv S\rv^{\frac{1}{2}} \l( \sum_{\theta \in S}\ell^2 \binom{i_{P''}(\theta)}{2} + \ell^2\lv S\rv \r)^{\frac{1}{2}} \\
    & \le \lv S\rv^{\frac{1}{2}}  \l( \ell^2 \lv \lo  (x,y) \in (P'')^2\colon \exists \theta \in S , \theta \text{ is incident to $x$ and $y$} \ro \rv +\ell^2 \lv S\rv \r)^{\frac{1}{2}} \\
    & \le \lv S\rv^{\frac{1}{2}} \l( \lv \lo \l( x_1,y_1,x_2,y_2 \r) \in A\times B\times A \times B \colon \exists \theta \in S , \theta y_1 =x_1 ,\theta y_2 =x_2 \ro \rv +\ell^2 \lv S\rv \r)^{\frac{1}{2}}\\
    & \le \lv S\rv^{\frac{1}{2}} \l( \lv \lo \l( x_1,y_1,x_2,y_2 \r) \in A\times B\times A \times B \colon x_1\cdot x_2^\perp =y_1\cdot y_2^\perp \ro \rv +\ell^2 \lv S\rv \r)^{\frac{1}{2}}\\
    & \ll \lv S\rv^{\frac{1}{2}} \l(  \frac{|A||B|}{p^{\frac{1}{2}}} + (pk|A||B|)^{\frac{1}{2}} + k\lv S\rv^{\frac{1}{2}} \r) ,
\end{align*} 
Putting these bounds together implies
\begin{align*}
    I(P,S) \lesssim \lv S\rv^{\frac{1}{2}} \l(  \frac{|A||B|}{p^{\frac{1}{2}}} + (pk|A||B|)^{\frac{1}{2}} + k\lv S\rv^{\frac{1}{2}} \r).
\end{align*}
as desired.
\end{proof}

\subsection{Incidence bounds for small sets (Theorem \ref{incidence-energy'})}

\subsubsection*{A skew dot-product energy estimate}
Given $B\subset \mathbb{F}_p^2$, this section is devoted to study the magnitude of the set 
\begin{equation}\label{ene-B}\{(x, y, u, v)\in B^4\colon x\cdot y^\perp=u\cdot v^\perp\}\end{equation}
when the size of $B$ is small. 

When the set $B$ is of large size, Lemma \ref{lmmay3} can be used to show that the number of such quadruples is almost the expected value $|B|^4/p$. However, when the size of $B$ is small, say, $|B|<p$, we have to deal with degenerate structures. 

More precisely, let $\ell$ be a line passing through the origin and $B$ be a subset of $\ell\setminus \{(0, 0)\}$, then the number of tuples $(x,y,u,v) \in B^4 $ such that $x\cdot y^\perp =u\cdot v^\perp $ is $\lv B\rv^4$. Note that if we remove the $'\perp'$ sign, then it will be at most $|B|^3$. 

This example shows that we need to have some conditions on the structures of $B$ so that a non-trivial estimate can be obtained.

If $B$ has a few distinct directions through the origin, then, by following  Rudnev's argument in \cite[Section 3]{R.14} identically, one obtains
\begin{lemma}\label{lem: 6.12}
Let $B$ be a subset in $\mathbb{F}_p^2$ with $|B|\le p$. Assume any line passing through the origin contains at most $k_1$ points from $B$ and $B$ determines at most $k_2$ distinct directions through the origin. The number of quadruples $(x, y, u, v)\in B^4$ such that $x\cdot y^\perp=u\cdot v^\perp$ is at most $k_1^{\frac{1}{2}}|B|^3+k_1k_2|B|^2+k_1^2|B|^2$.
\end{lemma}

If we only assume an assumption on the number of points on a line passing through the origin, then we have the following lemma.

\begin{lemma}\label{lem: 6.2}
Let $B$ be a subset in $\mathbb{F}_p^2$ with $|B|\le p^{\frac{8}{15}}$. Assume any line passing through the origin contains at most $k$ points from $B$. The number of quadruples $(x, y, u, v)\in B^4$ such that $x\cdot y^\perp=u\cdot v^\perp$ is at most $k^{\frac{4}{15}}|B|^{\frac{748}{225}}+ k^2 \lv B\rv^2$.
\end{lemma}

\subsubsection*{Proof of Lemma \ref{lem: 6.2}}
The key ingredient in the proof of Lemma \ref{lem: 6.2} is the multi-set version of a point-line incidence bound due to Stevens and De Zeeuw in \cite{SdZ}.
\begin{theorem}[Stevens--de~Zeeuw, \cite{SdZ}]\label{frank}
Let $P$ be a point set and $L$ a set of lines in $\mathbb{F}_p^2$. If $|P|\ll p^{\frac{8}{5}}$, then the number of incidences between $P$ and $L$, denoted by $I(P, L)$, satisfies 
\[I(P, L)\ll |P|^{\frac{11}{15}}|L|^{\frac{11}{15}}+|P|+|L|.\]
\end{theorem}
\begin{theorem}\label{eqn:IPLMS}
    Let $P$ be a multi-set of points and $L$ be a multi-set of lines in $\mathbb{F}_p^2$. We denote the set of distinct points in $P$ by $\overline{P}$ and the set of distinct lines in $L$ by $\overline{L}$. For $p\in \overline{P}$ and $\ell\in \overline{L}$, let $m(p)$ and $m(\ell)$ be the multiplicity of $p$ and $\ell$, respectively. If 
    $|P|=\sum_{p\in \overline{P}}m(p)\ll p^{\frac{8}{5}}$, then
\begin{equation}\label{eqan:IPLMS}
    I(P, L)\lesssim |P|^{\frac{7}{15}}|L|^{\frac{7}{15}}\left(\sum_{p\in \overline{P}}m(p)^2\right)^{\frac{4}{15}}\left(\sum_{\ell\in \overline{L}}m(\ell)^2\right)^{\frac{4}{15}}+|P|+|L|.
\end{equation}
\end{theorem}
\begin{proof}
    Our argument to prove this theorem is similar to proof of \cite[Lemma 2.12]{LuPe}. 

Let $L_k$ be the set of lines in $\overline{L}$ of multiplicity $\sim 2^k$, and $P_k$ be the set of points in $\overline{P}$ of multiplicity $\sim 2^k$. Set 
\[Q_1:=\sum_{p\in \overline{P}}m(p)^2 \quad \text{and}\quad ~Q_2:=\sum_{\ell\in \overline{L}}m(\ell)^2.\]

Then, it is clear that 
\[\sum_{k}2^k|P_k|=|P|,\quad ~\sum_{k}2^{2k}|P_k|=Q_1,\]
and
\[\sum_{k}2^k|L_k|=|L|,\quad ~\sum_{k}2^{2k}|L_k|=Q_2.\]
Thus, we have 
\[|P_k|\le \min \left\lbrace\frac{|P|}{2^k}, \frac{Q_1}{2^{2k}}\right\rbrace \quad \text{and} \quad |L_k|\le \min \left\lbrace\frac{|L|}{2^k}, \frac{Q_2}{2^{2k}}\right\rbrace.\]

We now observe 
\begin{align*}
    I(P, L)&=\sum_{i, j}I(P_i, L_j)=\sum_{2^j<Q_2/|L|}2^jI(P, L_j)+\sum_{2^j\ge Q_2/|L|}I(P, L_j)\\
    &=\sum_{\substack{2^i<Q_1/|P|\\2^j<Q_2/|L|}}2^i2^jI(P_i, L_j)+\sum_{\substack{2^i\ge Q_1/|P|\\2^j<Q_2/|L|}}2^i2^jI(P_i, L_j)\\&+\sum_{\substack{2^i<Q_1/|P|\\2^j\ge Q_2/|L|}}2^i2^jI(P_i, L_j)+\sum_{\substack{2^i\ge Q_1/|P|\\2^j\ge Q_2/|L|}}2^i2^jI(P_i, L_j)\\
    &=: I+II+III+IV.\\
\end{align*}
{\bf Bounding $I$:} Since $|P_i|\le |P|/2^i$ and $|L_j|\le |L|/2^j$, by Theorem \ref{frank}, we obtain 
\begin{align*}
I&\ll \sum_{\substack{2^i<Q_1/|P|\\
2^j<Q_2/|L|}}2^{i+j}\left(\frac{|P||L|}{2^{i+j}}\right)^{\frac{11}{15}}+|P|+|L|=\sum_{\substack{2^i<Q_1/|P|\\
2^j<Q_2/|L|}}2^{\frac{4(i+j)}{15}}(|P||L|)^{\frac{11}{15}}+|P|+|L|\\
&\lesssim |P|^{\frac{7}{15}}|L|^{\frac{7}{15}}(Q_1Q_2)^{\frac{4}{15}}+|P|+|L|.
\end{align*}

{\bf Bounding $II$:} Using $|P_i|\le Q_\frac{1}{2}^{2i}$ and $|L_j|\le |L|/2^j$, Theorem \ref{frank} implies
\begin{align*}
    II&:=\sum_{\substack{2^i\ge Q_1/|P|\\2^j<Q_2/|L|}}2^i2^jI(P_i, L_j)\le \sum_{\substack{2^i\ge Q_1/|P|\\2^j<Q_2/|L|}}2^i2^j \left(\frac{Q_1}{2^{2i}}\right)^{\frac{11}{15}}\left(\frac{|L|}{2^j}\right)^{\frac{11}{15}}+|P|+|L|\\
    &\lesssim |P|^{\frac{7}{15}}|L|^{\frac{7}{15}}(Q_1Q_2)^{\frac{4}{15}}+|P|+|L|.
\end{align*}

{\bf Bounding $III, IV$:} Similarly, we also have 
\[III, IV\lesssim |P|^{\frac{7}{15}}|L|^{\frac{7}{15}}(Q_1Q_2)^{\frac{4}{15}}+|P|+|L|.\]

This concludes the proof of the theorem.
\end{proof}
With this incidence bound, we are ready to prove Theorem \ref{lem: 6.2}.
\begin{proof}[Proof of Theorem \ref{lem: 6.2}]
We have
\begin{align*}
    \lv \{(x,y,u,v) \in B^4 \colon x\cdot y^\perp =u\cdot v^\perp \}\rv & = \lv \{(x,y,u,v) \in B^4 \colon x\cdot y^\perp =u\cdot v^\perp \ne 0\}\rv \\&+\lv \{(x,y,u,v) \in B^4 \colon x\cdot y^\perp =u\cdot v^\perp =0\}\rv \\
    & =: I +II
\end{align*}
Regarding $I$. For $a, u, v\in B$, $u\cdot v^\perp \ne 0$ we define $L_{a, u, v}$ to be the multi-set of lines of the form $x\cdot a^\perp=u\cdot v^\perp$. Let $L=\bigcup_{a, u, v\in B}L_{a, u, v}$. It is clear that the number of such quadruples is at most $I(B, L)$. We note that $|L|\le |B|^3$. By applying Theorem \ref{frank}, we observe that for each line in $L$, its multiplicity is at most $\ll k|B|^{\frac{22}{15}}$. For each $l \in L$, let $m(l)$ be the multiplicity $l$, we have
\[\sum_{\ell}m(\ell)=|L|\le |B|^3,\]
and 
\[\sum_{\ell}m(\ell)^2\le \max_{\ell}m(\ell)\cdot |B|^3\ll k|B|^{\frac{22}{15}}\cdot |B|^3.\]
It follows from Theorem \ref{eqn:IPLMS} for $B$ and $L$ that
\[ I(B,L) \lesssim \lv B\rv^{\frac{7}{15}} \lv B\rv^{\frac{21}{15}}\lv B\rv^{\frac{4}{15}}\l(  k\lv B\rv^{\frac{67}{15}} \r)^{\frac{4}{15}} +\lv B\rv +\lv L\rv \ll k^{\frac{4}{15}} \lv B\rv^{\frac{748}{225}}. \]

Regarding $II$. For each $x,u\in B$, there are at most $k$ points $y$ and $k$ points $v$ such that $x\cdot y^\perp =u\cdot v^\perp =0$. So 
\[ II \le k^2\lv B\rv^{2}. \]
This completes the proof.
\end{proof}

\subsubsection*{Proof of Theorem \ref{incidence-energy'}}
We follow the proof of Theorem \ref{thm:3.7} with Lemma \ref{lem: 6.12} and Lemma \ref{lem: 6.2} in place of Lemma \ref{lmmay3}, one has two following bounds
\begin{align*} I(P,S) &\lesssim  k_1^{\frac{1}{2}}|A|^{\frac{1}{2}}|S|+  \frac{k_1^{\frac{1}{2}}\lv B\rv^{\frac{1}{2}}|A|^{\frac{1}{2}}|S|^{\frac{3}{4}}}{p^{\frac{\epsilon}{4}}} + \frac{k_1^{\frac{1}{8}}\lv B\rv^{\frac{3}{4}}|A|^{\frac{1}{2}} |S|^{\frac{3}{4}}}{p^{\frac{\epsilon}{4}}}+\frac{k_1^{\frac{1}{4}}k_2^{\frac{1}{4}}\lv B\rv^{\frac{1}{2}} |A|^{\frac{1}{2}}|S|^{\frac{3}{4}}}{p^{\frac{\epsilon}{4}}},
\end{align*}
and 
\[ I(P,S) \lesssim  k_1^{\frac{1}{2}}|A|^{\frac{1}{2}}|S|+  \frac{k_1^{\frac{1}{2}}\lv B\rv^{\frac{1}{2}}|A|^{\frac{1}{2}}|S|^{\frac{3}{4}}}{p^{\frac{\epsilon}{4}}}+\frac{k_1^{\frac{1}{15}}|B|^{\frac{187}{225}} |A|^{\frac{1}{2}} |S|^{\frac{3}{4}}}{p^{\frac{\epsilon}{4}}}.\]
This completes the proof. $\square$

For applications, a comparison of these two incidence bounds, with $|B|\le p$, is provided as follows. 
\begin{enumerate}
    \item If $|B|\ge k_2k_1^{\frac{1}{2}}$, then
\[ I(P,S) \lesssim  k_1^{\frac{1}{2}}|A|^{\frac{1}{2}}|S|+  \frac{k_1^{\frac{1}{2}}\lv B\rv^{\frac{1}{2}}|A|^{\frac{1}{2}}|S|^{\frac{3}{4}}}{p^{\frac{\epsilon}{4}}}+ \frac{k_1^{\frac{1}{8}}\lv B\rv^{\frac{3}{4}}|A|^{\frac{1}{2}} |S|^{\frac{3}{4}}}{p^{\frac{\epsilon}{4}}}. \]
\item If $k_1^{\frac{165}{298}}k_2^{\frac{225}{298}} \le |B|< k_2k_1^{\frac{1}{2}}$, then
\[ I(P,S) \lesssim  k_1^{\frac{1}{2}}|A|^{\frac{1}{2}}|S|+  \frac{k_1^{\frac{1}{2}}\lv B\rv^{\frac{1}{2}}|A|^{\frac{1}{2}}|S|^{\frac{3}{4}}}{p^{\frac{\epsilon}{4}}}+ \frac{k_1^{\frac{1}{4}}k_2^{\frac{1}{4}}\lv B\rv^{\frac{1}{2}} |A|^{\frac{1}{2}}|S|^{\frac{3}{4}}}{p^{\frac{\epsilon}{4}}}. \]
\item If $|B| <k_1^{\frac{165}{298}}k_2^{\frac{225}{298}} $, then
\[ I(P,S) \lesssim  k_1^{\frac{1}{2}}|A|^{\frac{1}{2}}|S|+  \frac{k_1^{\frac{1}{2}}\lv B\rv^{\frac{1}{2}}|A|^{\frac{1}{2}}|S|^{\frac{3}{4}}}{p^{\frac{\epsilon}{4}}}+\frac{k_1^{\frac{1}{15}}|B|^{\frac{187}{225}} |A|^{\frac{1}{2}} |S|^{\frac{3}{4}}}{p^{\frac{\epsilon}{4}}}.\]
\end{enumerate}

\subsection{Alternative approach with relaxed conditions}
This section is devoted to prove the following analog, which offers a bound which is meaningful when the size of $S$ is large compared to the sizes of $A$ and $B$.

\begin{theorem}\label{thm: 6.2}
    Let $p$ be a prime and $P=A\times B \subseteq \l( \Fbb_p^2\times \Fbb_p^2 \r) \setminus \{ 0 \}$ with $\lv B\rv \le \lv A\rv \le p^{\frac{8}{15}}$ and any lines passing through the origin contains at most $k$ points from $B$. Then, we have 
    \[ I(P,S) \lesssim k^{\frac{2}{15}}\lv A\rv^{\frac{11}{15}}\lv B\rv^{\frac{209}{225}}  |S|^{\frac{1}{2}}  +k|A|^{\frac{1}{2}}|B|^{\frac{1}{2}}|S|^{\frac{1}{2}} +  k\lv S\rv .\]
\end{theorem}


Theorem \ref{thm: 6.2} is proved by following the proof of Theorem \ref{thm: 3.8} and the next lemma, whose proof is identical to that of Lemma \ref{lem: 6.2}.

\begin{lemma}\label{lem: quad}
    Let $A,B \subset \Fbb_p^2$ with $\lv B\rv \le \lv A \rv \le p^{\frac{8}{15}}$. Assume any line passing through the origin contains at most $k$ points from $B$ and at most $k$ points from $A$. The number of quadruples $(x,y,u,v) \in A\times A \times B \times B $ such that $x\cdot y^\perp =u\cdot v^\perp$ is at most $  k^{\frac{4}{15}}\lv A\rv^{\frac{22}{15}}\lv B\rv^{\frac{418}{225}}  + k^2\lv A\rv\lv B\rv$. 
\end{lemma}

\section{Proof of Theorems \ref{thm2} -- \ref{thm3'}}
\begin{proof}[Proof of Theorem \ref{thm2}]
On the one hand, we apply Theorem \ref{incidence} for $S$ and $P= S(E) \times E$ to obtain
\begin{align}\label{eq: 1}
    I(S(E)\times E,S) \ll \frac{\lv S(E)\rv \lv E\rv \lv S\rv}{p^2} + p\sqrt{\lv S(E)\rv \lv E\rv \lv S\rv} + \lv S\rv.
\end{align}
On the other hand, for each $x\in E$ and $\theta \in S$, there exists unique $y\in S(E)$ such that $(y,x)$ is incident to $\theta$. So, we obtain
 \[|E||S|=I(S(E)\times E,S) \ll \frac{\lv S(E)\rv \lv E\rv \lv S\rv}{p^2} + p\sqrt{\lv S(E)\rv \lv E\rv \lv S\rv} + \lv S\rv.\]
 Solving this inequality, the theorem follows. 
\end{proof}
Theorem \ref{thm555} will follow from the two following results. 
\begin{theorem}\label{thm1}
    Let $E\subset \mathbb{F}_p^2$ and $S\subset SL_2(\mathbb{F}_p)$.
\begin{itemize}
    \item[a.] If  $|E|\ge 4p$ and $|S|\gg p^2$, then there exists $x\in E$ such that $|S(E-x)|\gg p^2$. 
    \item[b.] If any line through the origin contains at most $k$ points from $E$, then 
     \[|S(E)|\gg \min \left\lbrace p^2, ~\frac{|S||E|}{pk} \right\rbrace.\]
\end{itemize}
\end{theorem}
Theorem \ref{thm1} (a) is not a part of Theorem \ref{thm555}, but we find it to be of independent interest. 

\begin{theorem}\label{thm5}
For any $\gamma\in (0, 1)$, there exists $\epsilon=\epsilon(\gamma)>0$ such that the following holds.

     Let $p$ be a sufficiently large prime. Let $E \subseteq \Fbb_p^2  \setminus \lo (0, 0)\ro$, and let $S$ be a symmetric subset of $\SL_2 (\Fbb_p )$ such that $p^\gamma < |S| < p^{3-2\gamma } $ and $ |S \cap gH| < p^{\frac{-\gamma}{2}}|S|$ for any subgroup $H \subsetneq \SL_2 (\Fbb_p) $ and $g\in \SL_2(\Fbb_p)$. Then, we have 
    \[ |S(E)| \gg \min \lo p^2 , \frac{|S|^{\frac{1}{2}}|E|}{p^{1-\frac{\epsilon}{2}}} \ro .\] Moreover,  if any line passing through the origin contains at most $k$ points from $E$, we have
    \[ |S(E)| \gg \min \lo p^2 ,\frac{|S|^{\frac{1}{2}}|E|}{p^{\frac{1-\epsilon}{2}}k^{\frac{1}{2}}} \ro .\]
\end{theorem}

To prove Theorem \ref{thm1}(a), we recall the following result from \cite{L.P.V.18}. 

\begin{lemma}[Corollary 10, \cite{L.P.V.18}]\label{exist}
    Let $E\subset \mathbb{F}_p^2$ with $|E|\ge 4p$. Then there exists a point $x\in E$ such that there are at least $p/2$ lines passing through $x$ and each line contains at least one other point from $E$. 
\end{lemma}

\begin{proof}[Proof of Theorem \ref{thm1}]\,
\medskip

{\bf Part a.}
    By Lemma \ref{exist}, there exists a point $x\in E$ such that there are at least $p/2$ lines passing through $x$ and each line contains at least one other point from $E$. From each of these lines, we pick one point which is different from $x$ and let $E'$ be the set of those points. Then we have $|E'|\ge p/2$. We observe that 
    \[|S(E-x)|\ge |S(E'-x)|.\]
    Note that $E'-x$ is a translation of $E'$ by $x$. So, any line passing through the origin contains at most one point from $E'-x$. Set $E''=E'-x$. We now estimate the size of $S(E'')$ from below. 

Set $P=E''\times S(E'')$ and $P'=\{\lambda\cdot u\colon u\in P, \lambda\ne 0\}$. We have $|P'|=(p-1)|P|$.
    Note that $I(P, S)=|E''||S|=\frac{1}{p-1}I(P', S)$. 
By Theorem \ref{incidence}, we have     
\[I(P', S)\ll \frac{|P'||S|}{p^2}+p\sqrt{|P'||S|} + \lv S\rv .\]
Putting the upper and lower bounds together, the theorem follows. 

{\bf Part b.} The proof is almost the same, we just need to partition the set $E$ into at most $k$ subsets $E_i$ such that each has the same structure as the set $E''$ in the Part a. So we omit the details.
\end{proof}

Theorem \ref{thm5}, Theorem \ref{thm345}, and Theorem \ref{thm3'} are proved by following the proof of Theorem \ref{thm2}, but we use Theorems \ref{thm: 4.12} and \ref{incidence-energy'} (1) in place of Theorem \ref{incidence}.

\begin{remark}
Theorem \ref{incidence-energy'} (2) implies the following bound
     \[ \lv S(E)\rv \gtrsim \min \lo \frac{\lv E\rv^{0.337} \lv S\rv^{\frac{1}{2}}p^{\epsilon /2}}{k_1^{\frac{2}{15}}}, \frac{\lv E\rv \lv S\rv^{\frac{1}{2}}p^{\epsilon /2}}{k_1}, \frac{\lv E\rv^2}{k_1}  \ro,\]
     which improves Theorem \ref{thm3'} in the range $|E|\le \min \lo p^{\frac{8}{15}} , k_1^{0.553}k_2^{0.755} \ro$. 
\end{remark}

\section{Incidence structures spanned by \texorpdfstring{$\mathbb{H}_1(\mathbb{F}_p)$}{H1(Fp)}}
As in the case of the special linear group, we define an incidence structure as follows. We say the point $(x, y)\in \mathbb{F}_p^3\times \mathbb{F}_p^3$ is incident to the matrix $\theta\in \mathbb{H}_1(\Fbb_p)$ if $\theta y=x$. 

\begin{theorem}\label{H-inci}Let $\epsilon\ge 0$.
    Let $X\subset \mathbb{H}_1(\Fbb_p)$ and $P=A\times B\subset \mathbb{F}_p^6$. Assume that all points in $A$ and $B$ have the third coordinate non-zero, and for each $(y, z)\in \mathbb{F}_p^2$, we have $|\pi_{23}^{-1}(y, z)\cap B|\le p^{1-\epsilon}$, then the number of incidences between $X$ and $P$ satisfies
    \[\left\vert I(P, X)- \frac{|P||X|}{p^3} \right\vert\le p^{\frac{3-\epsilon}{2}}|P|^{\frac{1}{2}}|X|^{\frac{1}{2}}.\]
\end{theorem}


We now turn our attention to proving Theorem \ref{H-inci}.
\begin{proof}[Proof of Theorem \ref{H-inci}]
By repeating the argument as in the case of $SL_2(\mathbb{F}_p)$, we have 
    \begin{align*}
        I(P,X) & = \sum_{\substack{(x,y )\in P \\ \theta \in X}} 1_{\theta x=y } =\frac{1}{p^3} \sum_{m\in \Fbb_p^3} \sum_{\substack{(x,y)\in P, \\ \theta \in X}} \chi \l( m\cdot \l( \theta y-x \r) \r) \\
        & = \frac{\lv P\rv \lv X\rv}{p^3} + \frac{1}{p^3} \sum_{ m\in \Fbb_p^3 \setminus \lo (0,0, 0)\ro } \sum_{\substack{ (x,y)\in P ,\\ \theta \in X }} \chi \l( m \cdot \l( \theta y -x \r) \r) \\
        & = \frac{\lv P\rv \lv X\rv}{p^3} + p^3\sum_{m\ne (0, 0, 0)} \sum_{\theta \in X} \widehat{P} \l( -m,\theta^t m\r).
    \end{align*}
By using the Cauchy--Schwarz inequality, one has 
\begin{align*}
    &\sum_{\theta\in X}\sum_{m\ne (0, 0, 0)}\widehat{P}(-m, \theta^tm)\le |X|^{\frac{1}{2}}\left(\sum_{\theta\in \mathbb{H}_1(\Fbb_p)}\sum_{m_1, m_2\ne (0, 0, 0)}\widehat{P}(-m_1, \theta^tm_1)\overline{\widehat{P}(-m_2, \theta^tm_2)}\right)^{\frac{1}{2}}.
\end{align*}
We now observe 
\begin{align*}
    &\sum_{\theta\in \mathbb{H}_1(\Fbb_p)}\sum_{m_1, m_2\ne (0, 0, 0)}\widehat{P}(-m_1, \theta^tm_1)\overline{\widehat{P}(-m_2, \theta^tm_2)}\\&=\frac{1}{p^{12}}\sum_{\theta}\sum_{m_1, m_2\ne (0, 0, 0)}\sum_{(x_1, y_1), (x_2, y_2)}P(x_1, y_1)P(x_2, y_2)\chi(-m_1x_1+\theta^tm_1y_1)\chi(m_2x_2-\theta^tm_2y_2)\\
    &=\sum_{m_1, m_2}-\sum_{m_1=(0, 0, 0), m_2\ne (0, 0, 0)}-\sum_{m_1\ne (0, 0, 0), m_2=(0, 0, 0)}-\sum_{m_1=m_2=(0, 0, 0)}\\
    &=I-II-III-IV.
\end{align*}
We now estimate each term separately. 

Regarding $IV$, it is clear that it is equal to $|P|^2/p^9$. The terms $II$ and $III$ are the same, and we can see that 
\[II=\frac{|P|}{p^{12}}\sum_{\theta}\sum_{m_2\ne (0, 0, 0)}\sum_{(x_2, y_2)}P(x_2, y_2)\chi(m_2(x_2-\theta y_2)),\]
which can be written as 
\[\frac{|P|(p^2-1)}{p^{12}}\sum_{\theta}\sum_{(x_2, y_2)\in P, x_2=\theta y_2} 1-\frac{|P|}{p^{12}}\sum_{\theta}\sum_{(x_2, y_2)\in P, x_2\ne\theta y_2}1.\]
So, $-II\le \frac{|P|^2}{p^9}$.
Regarding $I$,
\begin{align*}
    &I=\frac{1}{p^{12}}\sum_{\theta}\sum_{m_1, m_2\in \mathbb{F}_p^3}\sum_{(x_1, y_1), (x_2, y_2)}P(x_1, y_1)P(x_2, y_2)\chi(m_1(\theta y_1-x_1))\chi(m_2(\theta y_2-x_2))\\
    &=\frac{1}{p^6}\sum_{\theta}\sum_{(x_1, y_1), (x_2, y_2)}P(x_1, y_1)P(x_2, y_2)1_{\theta y_1=x_1}1_{\theta y_2=x_2}.\\
\end{align*}
To proceed further, we need to count the number of quadruples $(x, y, z)\in B, (x', y', z')\in A, (u, v, w)\in B, (u', v', w')\in A$ such that there exists $\theta=[a, b, c]\in \mathbb{H}_1(\Fbb_p)$ and $\theta (x, y, z)=(x', y', z')$, ~$\theta(u, v, w)=(u', v', w')$.

For such quadruples, one has 
\[yw'+zv'=y'w+vz', ~z(u'-u)+w(x-x')=a(vz-wy), ~z=z', w=w'.\]
From here, as in the case of the special linear group, we need to make use of a number of preliminary lemmas, which will be proved later.
\begin{proposition}\label{propo2.1}
Given $A, B\subset \mathbb{F}_p^3$, let $N$ be the number of quadruples $(x, y, z)\in A, (x', y', z')\in B, (u, v, w)\in A, (u', v', w')\in B$ such that 
\[yw'+zv'=y'w+vz', ~z=z', w=w'.\]
Suppose in addition that for each $(y, z)\in \mathbb{F}_p^2$, we have $\pi_{23}^{-1}(y, z)\cap A$ or $\pi_{23}^{-1}(y, z)\cap B$ are of sizes at most $p^{1-\epsilon}$, then we have 
\[N\le \frac{|A|^2|B|^2}{p}+p^{3-2\epsilon}|A||B|.\]
\end{proposition}
\begin{proposition}\label{propo2.2}
    Given $A, B\subset \mathbb{F}_p^3$, let $N'$ be the number of quadruples $(x, y, z)\in B, (x', y', z')\in A, (u, v, w)\in B, (u', v', w')\in A$ such that 
\[yw'+zv'=y'w+vz', ~vz-wy=0, ~zu'+xw'-z'u-x'w=0, ~w=w',~z=z'.\]
Then $N'\le p|A||B|$.
\end{proposition}
\begin{lemma}\label{lm23}
    Given $(x, y, z), (x', y', z'), (u, v, w), (u', v', w')$ in $\mathbb{F}_p^3$ with $w\ne 0$, $z\ne 0$, and $yw'+zv'=y'w+z'v$. Let $N''$ be the number of matrices $\theta=[a, b, c]\in \mathbb{H}_1(\Fbb_p)$ such that 
\[\theta(x, y, z)=(x', y', z')\]
and 
\[\theta(u, v, w)=(u', v', w').\]
\begin{itemize}
    \item If $vz-wy\ne 0$, then $N''=1$. 
    \item If $vz-wy=0$ and $z(u'-u)=w(x-x')$, then $N''=p$.

\item If $vz-wy=0$ and $z(u'-u)\ne w(x-x')$, then $N''=0$.
\end{itemize}
\end{lemma}

So, combining Propositions \ref{propo2.1}, \ref{propo2.2}, and Lemma \ref{lm23} gives
\[I\le \frac{|P|^2}{p^3}+p^{3-\epsilon}|P|.\]
In other words,
\[I-II-III-IV\ll \frac{1}{p^6}\cdot p^{3-\epsilon}|P|,\]
and 
 \[\left\vert I(P, X)- \frac{|P||X|}{p^3} \right\vert\le p^{\frac{3-\epsilon}{2}}|P|^{\frac{1}{2}}|X|^{\frac{1}{2}}.\]
 This completes the proof.
\end{proof}

The rest of this section is devoted to prove Propositions \ref{propo2.1}, \ref{propo2.2}, and Lemma \ref{lm23}.

We first start with Proposition \ref{propo2.2} and Lemma \ref{lm23}, since they are straightforward. 

\begin{proof}[Proof of Proposition \ref{propo2.2}]
    By fixing $(x, y, z)\in A$ and $(u', v', w')\in B$, then $y'$ and $v$ are determined uniquely by 
    \[yw'+zv'=y'w+vz', ~vz-wy=0, ~zu'+xw'-z'u-x'w=0.\]
With each $u\in \mathbb{F}_p$, $x'$ is determined uniquely. Thus, $N'\le p|A||B|$.
\end{proof}
\begin{proof}[Proof of Lemma \ref{lm23}]
A direct computation shows that 
\[b=\frac{v'-v}{w}=\frac{y'-y}{z}.\]
Note that $z(u'-u)+w(x-x')=a(vz-wy)$. If $vz-wy\ne 0$, then 
\[a=\frac{z(u'-u)+w(x-x')}{vz-wy}.\]

For such $a$, $c$ is determined uniquely. 

If $vz-wy=0$ and $z(u'-u)=w(x-x')$, there are $p$ such matrices $\theta$. 

If $vz-wy=0$ and $z(u'-u)\ne w(x-x')$, there is no such matrix $\theta$. 
\end{proof}

We now move to the proof of Proposition \ref{propo2.1}.

By repeating the proof of \cite[Theorem 2.1]{Hetal}, the following weighted version can be obtained. 
\begin{lemma}\label{lm3.3}
    Let $U, V$ be two sets in $\mathbb{F}_p^4$, and $F\colon U\to \mathbb{N}$, $G\colon V\to \mathbb{N}$. 
   Define 
   \[M=\sum_{\substack{u\in U, v\in V\\ u\cdot v=0}}F(u)G(v).\]Then we have 
    \begin{align*}
    &\left\vert M-\frac{(\sum_{u\in U}F(u))(\sum_{v\in V}G(v))}{p}\right\vert\le  p^2\left(\sum_{u\in U}|F(u)|^2\right)^{\frac{1}{2}}\cdot \left(\sum_{v\in V}|G(v)|^2\right)^{\frac{1}{2}}.     
    \end{align*}
\end{lemma}

In our setting, we partition the set $A$ into $A_\lambda$ and the set $B$ into $B_\lambda$, for $\lambda\in \mathbb{F}_p$, such that any two points in $A_\lambda$ have the same last coordinate which is equal to $\lambda$, and the same applies to $B_\lambda$. 

Recall $\pi_{23}\colon \mathbb{F}_p^3\to \mathbb{F}_p^2$ be the projection onto the two last coordinates. 

For $A_\lambda\subset \mathbb{F}_p^3$, we define $\overline{A_\lambda}=\pi_{23}(A_\lambda)$ . For each $x\in \overline{A_\lambda}\subset \mathbb{F}_p^2$, define $f(x)=\pi_{23}^{-1}(x)\cap A_\lambda$.


For $B_\lambda\subset \mathbb{F}_p^3$, we define $\overline{B_\lambda}=\pi_{23}(B_\lambda)$ . For each $x\in \overline{B_\lambda}$, define $g(x)=\pi_{23}^{-1}(x)\cap B_\lambda$.

Let $N_{\lambda, \beta}$ be the number of quadruples $(a, a', b, b')\in \overline{A_\lambda}\times \overline{B_\lambda}\times\overline{A_\beta}\times \overline{B_\beta}$ such that $a\cdot b'=a'\cdot b$, counted with multiplicity, i.e. 
\[N_{\lambda, \beta}=\sum_{a, b, a', b'\colon a\cdot b'=a'\cdot b}f(a)g(a')f(b)g(b')\]

We want to bound $N_{\lambda, \beta}$ from above. The next result will be proved by using the above lemma.
\begin{lemma}\label{weight-H}
    For fixed $\lambda, \beta$, we have 
    \begin{align*}
         &N_{\lambda, \beta}\le \frac{1}{p}\l(\sum_{x\in \overline{A_\lambda}}f(x)\r) \l(\sum_{x\in \overline{A_\beta}}f(x)\r) \l(\sum_{y\in \overline{B_\lambda}}g(y)\r) \l(\sum_{y\in \overline{B_\beta}}g(y)\r)\\
         &+p\l(\sum_{x\in \overline{A_\lambda}}|f(x)|^{2}\r)^{\frac{1}{2}}\l(\sum_{x\in \overline{A_\beta}}|f(x)|^2\r)^{\frac{1}{2}}\l(\sum_{y\in \overline{B_\lambda}}|g(y)|^2\r)^{\frac{1}{2}}\l(\sum_{y\in \overline{B_\beta}}|g(y)|^2\r)^{\frac{1}{2}}.
    \end{align*}
\end{lemma}
\begin{proof}
To prove this lemma, we first enlarge the set $A_\lambda, A_\beta, B_\lambda$, and $B_\beta$. More precisely, define 
\[U=\{t(a, a')\colon a\in \overline{A_{\lambda}}, ~a'\in \overline{B_{\lambda}},~ t\ne 0\},\]
and
\[V=\{t(b, b')\colon b\in \overline{A_{\beta}}, ~b'\in \overline{B_{\beta}},~ t\ne 0\}.\]
Also define 
\[F(t(a, a'))=f(a)g(a'), ~~G(t(b, b'))=f(b)g(b')\]
for all $t(a, a')\in U$ and $t(b, b')\in V$.

If $(a, a', b, b')\in \overline{A_\lambda}\times \overline{B_\lambda}\times\overline{A_\beta}\times \overline{B_\beta}$ such that $a\cdot b'=a'\cdot b$, then 
\[(a, a')\cdot (-b', b)=0,\]
which implies
\[t(a, a')\cdot t'(-b', b)=0\]
for all $t, t'\ne 0$.

Therefore, 
\[N_{\lambda, \beta}=\frac{1}{(p-1)^2}M,\]
where 
\[M=\sum_{u\in U, v\in V, u\cdot v=0}F(u)G(v).\]
Applying Lemma \ref{lm3.3} gives us 
\begin{align*}
         &N_{\lambda, \beta}\le \frac{1}{p}\l(\sum_{x\in \overline{A_\lambda}}f(x)\r) \l(\sum_{x\in \overline{A_\beta}}f(x)\r) \l(\sum_{y\in \overline{B_\lambda}}g(y)\r) \l(\sum_{y\in \overline{B_\beta}}g(y)\r)\\
         &+p\l(\sum_{x\in \overline{A_\lambda}}|f(x)|^{2}\r)^{\frac{1}{2}}\l(\sum_{x\in \overline{A_\beta}}|f(x)|^2\r)^{\frac{1}{2}}\l(\sum_{y\in \overline{B_\lambda}}|g(y)|^2\r)^{\frac{1}{2}}\l(\sum_{y\in \overline{B_\beta}}|g(y)|^2\r)^{\frac{1}{2}}.
    \end{align*}
This completes the proof. 
\end{proof}
Using this lemma, we take the sum over all pairs $(\lambda, \beta)$, and obtain 
\[N\le \frac{|A|^2|B|^2}{p}+p \left(\sum_{x\in \mathbb{F}_p^2}f(x)^2\right)\cdot \left(\sum_{y \in \mathbb{F}_p^2}g(y)^2\right),\]
by using the Cauchy--Schwarz inequality. 

The trivial upper bound that $f(x), g(y)\le p$ gives 
\[N\le \frac{|A|^2|B|^2}{p}+p^3|A||B|.\]
If we assume in addition that either $\max_{x}f(x)\le p^{1-\epsilon}$ or $\max_{x}g(x)\le p^{1-\epsilon}$, then 
\[N\le \frac{|A|^2|B|^2}{p}+p^{3-\epsilon}|A||B|.\]
This completes the proof. $\square$

\section{Proof of Theorem \ref{main-H}}
With Theorem \ref{H-inci} in hand, one can prove Theorem \ref{main-H} easily. To see this, we set $B=E$ and $A=X(E)$. We now estimate the number of incidences between $A\times B$ and $X$. It is clear that $I(A\times B, X)=|X||B|$. Applying Theorem \ref{H-inci} on the number of incidences between $A\times B$ and $X$, we obtain 
\[I(A\times B, X)\le \frac{|A||B||X|}{p^3}+p^{\frac{3-\epsilon}{2}}|X|^{\frac{1}{2}}|A|^{\frac{1}{2}}|B|^{\frac{1}{2}}.\]
Putting the lower and upper bounds together and solving the inequality give us the desired bound. 

\section{Acknowledgements}
Thang Pham would like to thank the Vietnam Institute for Advanced Study in Mathematics (VIASM) for the hospitality and for the excellent working condition.

Norbert Hegyv\'{a}ri was supported by the National Research, Development and Innovation Office NKFIH Grant No K-146387. Alex Iosevich was partially supported by the National Science Foundation,  NSF DMS 2154232. Thang Pham was partially supported by ERC Advanced Grant no. 882971, ``GeoScape", and by the Erd\H{o}s Center.

\end{document}